\documentclass[11pt,amssymb,amsfont,a4paper]{article}
\usepackage{latexsym,amssymb,amsmath,amsthm,color}
\usepackage{geometry}
\usepackage{url}
\usepackage{graphicx}
\usepackage[utf8]{inputenc}
\usepackage[numbers]{natbib}
\usepackage{authblk}
\usepackage{tikz}
\usepackage{arydshln}

\theoremstyle{definition}
\newtheorem{definition}{Definition}
\newtheorem{example}[definition]{Example}

\theoremstyle{plain}
\newtheorem{theorem}{Theorem}
\newtheorem{proposition}[definition]{Proposition}
\newtheorem{lemma}[definition]{Lemma}

\newtheorem{corollary}[definition]{Corollary}

\title{Theory and applications of \\linearized multivariate skew polynomials}
\author{Umberto Mart{\'i}nez-Pe\~{n}as \thanks{umberto.martinez@unine.ch}}
\affil{Institute of Computer Science and Mathematics, \\ University of Neuch{\^a}tel, Switzerland}

\date{}

\begin{document}

\maketitle

\begin{abstract}
In this work, linearized multivariate skew polynomials over division rings are introduced. Such polynomials are right linear over the corresponding centralizer and generalize linearized polynomial rings over finite fields, group rings or differential polynomial rings. Their natural evaluation is connected to the remainder-based evaluation of free multivariate skew polynomials. It is shown that P-independent sets are those given by right linearly independent sets when partitioned into conjugacy classes. Hence finitely generated P-closed sets correspond to lists of finite-dimensional right vector spaces, extending Lam and Leroy's results on univariate skew polynomials. It is also shown that products of free multivariate skew polynomials translate into coordinate-wise compositions of linearized multivariate skew polynomials, which in turn translate into matrix products over the corresponding centralizers. Later, linearized multivariate Vandermonde matrices are introduced, which generalize multivariate Vandermonde, Moore and Wronskian matrices. The previous results explicitly give their ranks in general. P-Galois extensions of division rings are then introduced, which generalize classical (finite) Galois extensions. Three Galois-theoretic results are generalized to such extensions: Artin's theorem on extension degrees, the Galois correspondence and Hilbert's Theorem 90.


\end{abstract}

\section{Introduction} \label{sec intro}

Univariate skew polynomials were introduced by Ore in \cite{ore}. A natural definition of evaluation on them, via Euclidean division, was later given by Lam and Leroy in \cite{lam, lam-leroy}. Thanks to this concept of evaluation, Lam and Leroy introduced the notion of \textit{P-independence} of evaluation points in \cite{lam, algebraic-conjugacy}, which in turn gives rise to the concept of \textit{P-closed set} and \textit{P-basis} of a P-closed set. Intuitively, a finite set of evaluation points is P-independent if we may perform Lagrange interpolation on them (see Theorem \ref{th lagrange interpolation}). The main characterization of P-independent sets is also due to Lam and Leroy, and states that P-independent sets are those such that, when partitioned into conjugacy classes \cite[Th. 23]{lam}, each subset of the partition is given by a linearly independent set over the corresponding centralizer \cite[Th. 4.5]{lam-leroy}. In this way, P-closed sets correspond to lists of vector spaces and P-independent sets form a representable matroid \cite{oxley}.

With this characterization at hand, Lam and Leroy explicitly described the ranks of matrices obtained by evaluating (univariate) skew polynomials, which generalize Vandermonde matrices \cite{lam} and are related to Moore matrices \cite{moore} \cite[Lemma 3.51]{lidl} and Wronskian matrices \cite[Def. 1.11]{singer}. Such results were used to prove that skew Reed-Solomon codes are MDS codes \cite{skew-evaluation1} and linearized Reed-Solomon codes are maximum sum-rank distance (MSRD) codes \cite{linearizedRS}. Linearized Reed-Solomon codes have provided PMDS codes with order-optimal finite field sizes \cite{cai-field, gopi-field}, sought after by researchers at Microsoft. Such codes and the MSRD property (critical for such applications) were obtained by an alternative evaluation of skew polynomials \cite[Lemma 24]{linearizedRS}, which turns them into linear maps, recovering linearized polynomial functions \cite{orespecial} \cite[Ch. 3]{lidl} as a particular case. 

In Section \ref{sec main definitions}, we extend such a \textit{linearized} evaluation to (free) \textit{multivariate skew polynomials} \cite{multivariateskew}, thus defining \textit{linearized multivariate skew polynomials}, which generalize linearized polynomials over finite fields \cite{orespecial} \cite[Ch. 3]{lidl}, group rings \cite[p. 104]{lang} and differential polynomials \cite[App. D]{singer}. As in the univariate case \cite[Lemma 24]{linearizedRS}, we connect this type of evaluation with the arithmetic evaluation of multivariate skew polynomials based on Euclidean divisions (Theorem \ref{theorem evaluation of skew as map}). Later in Section \ref{sec linearized P-closed sets}, we will use this connection and Lagrange interpolation \cite[Th. 4]{multivariateskew} to extend the important results \cite[Th. 23]{lam} and \cite[Th. 4.5]{lam-leroy} to the multivariate case (Theorems \ref{theorem linearized version of P-closed in one conj} and \ref{theorem lin P-closed sets in several conj classes}).

In Section \ref{sec skew and lin polynomial arithmetic}, we turn to the arithmetic of the introduced linearized multivariate skew polynomials. We show that: (1) On one conjugacy class, products of multivariate skew polynomials are mapped onto compositions of linearized multivariate skew polynomials (Theorem \ref{th skew pol is composition}) and matrix products (Theorem \ref{th composition is matrix multiplication}); and (2) Over several conjugacy classes, products of multivariate skew polynomials decompose into coordinate-wise compositions of linearized multivariate skew polynomials and products of matrices (Theorem \ref{th product decompositions}).

In Section \ref{sec generalizations Vandermonde, Moore, Wronskian}, we define linearized multivariate Vandermonde matrices, which generalize multivariate Vandermonde, Moore and Wronskian matrices over division rings. We connect them to skew multivariate Vandermonde matrices \cite{multivariateskew} and provide an explicit formula for their ranks (Theorem \ref{th rank of lin vandermonde matrix}). 

In Section \ref{sec generalizations galois and derivations}, we define P-Galois extensions of division rings, which generalize classical Galois extensions. We generalize three classical results in Galois theory: (1) Artin's Theorem \cite[Th. 14]{artin-lectures} on extension degrees; (2) The Galois correspondence \cite[Th. 16]{artin-lectures}; and (3) Hilbert's Theorem 90 \cite[Th. 21]{artin-lectures} \cite[Th. 90]{hilbert}.

As an additional application, two new families of error-correcting codes could be defined and studied using the results in this manuscript, with the idea of generalizing classical Reed-Muller codes. First, skew Reed-Muller codes could be defined using remainder-based evaluations (Definition \ref{def skew evaluation}). Second, linearized Reed-Muller codes could be defined using linearized evaluations (Definition \ref{def lin evaluation}). These two code families would generalize, respectively, skew and linearized Reed-Solomon codes \cite{skew-evaluation1, linearizedRS}. Furthermore, both families would generalize classical Reed-Muller codes. They would also intersect with the recent Reed-Muller type codes from \cite{AugotSkewRM, skewRM}, but they would constitute different code families in general. A complete study of such skew and linearized Reed-Muller codes is left open.

\textbf{Notation:} $ \mathcal{A}^{m \times n} $ denotes the set of $ m \times n $ matrices over $ \mathcal{A} $. We also denote $ \mathcal{A}^n = \mathcal{A}^{n \times 1} $. $ \mathcal{B}^\mathcal{A} $ denotes the set of all maps $ \mathcal{A} \longrightarrow \mathcal{B} $. We will fix a division ring $ \mathbb{F} $ throughout the manuscript. Given a ring $ \mathcal{R} $, we denote by $ (\mathcal{A}) \subseteq \mathcal{R} $ the left ideal generated by $ \mathcal{A} \subseteq \mathcal{R} $. We denote by $ \langle \mathcal{B} \rangle^L $ and $ \langle \mathcal{B} \rangle^R $ the left and right $ \mathbb{F} $-linear vector space generated by $ \mathcal{B} $, respectively. We also denote by $ \dim^L_\mathbb{F} $ and $ \dim^R_\mathbb{F} $ left and right dimensions over $ \mathbb{F} $. Rings are not assumed to be commutative, but all of them will be assumed to have a multiplicative identity, and all ring morphisms preserve multiplicative identities.

\section{Main definitions and the natural evaluation maps} \label{sec main definitions}

In this section, we define linearized multivariate skew polynomials. We extend the notion of centralizers from the univariate case \cite[Eq. (3.1)]{lam-leroy} to the multivariate case, and we show that linearized multivariate skew polynomials are right linear over the corresponding centralizer. Finally, we translate the natural evaluation on linearized multivariate skew polynomials to evaluations by Euclidean division \cite[Def. 9]{multivariateskew}.
%

\begin{definition}[\textbf{Free multivariate skew polynomials \cite{multivariateskew}}] \label{def free skew polynomials}
Given a ring morphism $ \sigma : \mathbb{F} \longrightarrow \mathbb{F}^{n \times n} $, we say that $ \delta : \mathbb{F} \longrightarrow \mathbb{F}^n $ is a $ \sigma $-derivation if
$$ \delta(a+b) = \delta(a) + \delta(b) \quad \textrm{and} \quad \delta(ab) = \sigma(a) \delta(b) + \delta(a) b, $$
for all $ a,b \in \mathbb{F} $. Let $ x_1, x_2, \ldots, x_n $ be $ n $ pair-wise distinct variables, and denote by $ \mathcal{M} $ the free (non-commutative) monoid on such variables, whose elements are called monomials and whose monoid identity is denoted by $ 1 $. The \textit{free skew polynomial ring} over $ \mathbb{F} $ in the variables $ x_1, x_2, \ldots, x_n $ with morphism $ \sigma $ and derivation $ \delta $ is the left vector space $ \mathbb{F}[\mathbf{x}; \sigma, \delta] $ with basis $ \mathcal{M} $ and with product given by appending monomials and
\begin{equation}
\mathbf{x} a = \sigma(a) \mathbf{x} + \delta(a),
\label{eq rule defining product}
\end{equation}
for $ a \in \mathbb{F} $, where $ \mathbf{x} = (x_1, x_2, \ldots, x_n)^T \in \mathcal{M}^n $. Therefore, (\ref{eq rule defining product}) is a short form of
\begin{equation}
x_i a = \sum_{j=1}^n \sigma_{i,j}(a) x_j + \delta_i(a),
\label{eq rule defining product expanded}
\end{equation}
for $ i = 1,2, \ldots, n $, where $ \sigma_{i,j} $ and $ \delta_i $ denote the components of $ \sigma $ and $ \delta $, respectively. Each $ F \in \mathbb{F}[\mathbf{x}; \sigma, \delta] $ is called a \textit{skew polynomial} and can be uniquely written as 
\begin{equation}
F = \sum_{\mathfrak{m} \in \mathcal{M} } F_\mathfrak{m} \mathfrak{m},
\label{eq general form of skew polynomial}
\end{equation}
where $ F_\mathfrak{m} \in \mathbb{F} $, for $ \mathfrak{m} \in \mathcal{M} $, which are all zero except for a finite number of them. Define the degree of a monomial $ \mathfrak{m} \in \mathcal{M} $ as its length as a string, and define the degree of a non-zero skew polynomial $ F \in \mathbb{F}[\mathbf{x}; \sigma, \delta] $, denoted by $ \deg(F) $, as the maximum degree of a monomial $ \mathfrak{m} \in \mathcal{M} $ such that $ F_\mathfrak{m} \neq 0 $. We also define $ \deg(0) = - \infty $.
\end{definition}

With the product (\ref{eq rule defining product}), $ \mathbb{F}[\mathbf{x}; \sigma, \delta] $ is a ring and a left vector space over $ \mathbb{F} $ such that
\begin{equation}
\deg(F + G) \leq \max \{ \deg(F), \deg(G) \} \quad \textrm{and} \quad \deg(FG) = \deg(F) + \deg(G),
\label{eq product subadditive on degrees}
\end{equation}
for all $ F,G \in \mathbb{F}[\mathbf{x}; \sigma, \delta] $. In fact, by \cite[Th. 1]{multivariateskew}, all the products satisfying (\ref{eq product subadditive on degrees}) in the set $ \mathbb{F}[\mathbf{x}; \sigma, \delta] $ are given by (\ref{eq rule defining product}). Furthermore, if $ {\rm Id} : \mathbb{F} \longrightarrow \mathbb{F}^{n \times n} $ is the ring morphism given by $ {\rm Id}(a) = a I $, for $ a \in \mathbb{F} $, where $ I \in \mathbb{F}^{n \times n} $ is the $ n \times n $ identity matrix, then $ \mathbb{F}[\mathbf{x}; {\rm Id}, 0] $ is the free conventional polynomial ring in the variables $ x_1, x_2, \ldots, x_n $, as in \cite[Sec. 0.11]{cohn} and \cite[Ex. (1.2)]{lam-book}. To the best of our knowledge, the results in this paper are also new for free conventional polynomial rings over a division ring (the case $ \sigma = {\rm Id} $, $ \delta = 0 $). 

The ring $ \mathbb{F}[\mathbf{x}; \sigma, \delta] $ can also be characterized by the following universal property.

\begin{lemma} \label{lemma universal property}
Let $ \mathcal{R} $ be a ring and a left vector space over $ \mathbb{F} $, whose product is left $ \mathbb{F} $-linear in the first component. Assume that there exist elements $ y_1, y_2, \ldots, y_n \in \mathcal{R} $ satisfying
$$ y_i a = \sum_{j=1}^n \sigma_{i,j}(a) y_j + \delta_i(a), $$
for $ i = 1,2, \ldots, n $, for all $ a \in \mathbb{F} $. Then there exists a unique left $ \mathbb{F} $-linear ring morphism $ \varphi : \mathbb{F}[\mathbf{x}; \sigma, \delta] \longrightarrow \mathcal{R} $ such that $ \varphi(x_i) = y_i $, for $ i = 1,2, \ldots, n $.
\end{lemma}

It was shown in \cite[Th. 2]{skew-class} that all ring morphisms $ \sigma : \mathbb{F} \longrightarrow \mathbb{F}^{n \times n} $ are diagonalizable if $ \mathbb{F} $ is a finite field. However, this is far from the case in general.

\begin{example} \label{ex wild example II}
Let $ \mathbb{F} = \mathbb{F}_4 (z) $, where $ \mathbb{F}_4 $ is the finite field with $ 4 $ elements, $ z $ is transcendental, and let $ \gamma \in \mathbb{F}_4 $ be a primitive element (i.e., $ \mathbb{F}_4 = \{ 0 , 1, \gamma, \gamma^2 \} $). Define
$$ Z = \left( \begin{array}{cc}
0 & z \\
\gamma z & 0
\end{array} \right) \in \mathbb{F}_4 (z)^{2 \times 2} . $$
The unique ring morphism $ \sigma_Z : \mathbb{F}_4 (z) \longrightarrow \mathbb{F}_4 (z)^{2 \times 2} $ such that $ \sigma_Z(z) = Z $ is given by
$$ \sigma_Z(f(z)) = \left( \begin{array}{cc}
\tau (f(z)) & \gamma \partial (f(z)) \\
\gamma^2 \partial (f(z)) & \tau(f(z))
\end{array} \right) \in \mathbb{F}_4 [z]^{2 \times 2}, $$
where $ \tau(f(z)) \in \mathbb{F}_4 [z] $ and $ \partial(f(z)) \in \mathbb{F}_4 [z] $ are formed by the even and odd terms in $ f(\gamma^2 z) \in \mathbb{F}_4 [z] $, respectively, for $ f(z) \in \mathbb{F}_4 [z] $. The subfield of elements in $ \mathbb{F}_4 (z) $ fixed by $ \sigma_Z $ is $ \mathbb{F}_4 (z^6) $. Therefore, $ \sigma_Z $ is neither diagonalizable nor triangulable.
\end{example}

The motivation behind free skew polynomials as in Definition \ref{def free skew polynomials} is that they admit a natural arithmetic evaluation map \cite[Def. 9]{multivariateskew}, which is guaranteed by unique-remainder Euclidean division \cite[Lemma 5]{multivariateskew}. In contrast, \textit{iterated skew polynomial rings}  lack uniqueness for such a type of evaluation \cite[Remark 8]{multivariateskew}.

\begin{definition} [\textbf{Skew evaluation \cite{multivariateskew}}] \label{def skew evaluation}
For $ \mathbf{a} = (a_1, a_2, \ldots, a_n) \in \mathbb{F}^n $ and a skew polynomial $ F \in \mathbb{F} [\mathbf{x} ; \sigma,\delta] $, we define its evaluation on $ \mathbf{a} $, denoted by $ F(\mathbf{a}) = E_{\mathbf{a}}^S(F) $, as the unique element $ F(\mathbf{a}) \in \mathbb{F} $ satisfying
$$ F - F(\mathbf{a}) \in \left( x_1-a_1, x_2 - a_2, \ldots, x_n - a_n \right). $$
Given $ \Omega \subseteq \mathbb{F}^n $, we define the skew evaluation map over $ \Omega $ as the left $ \mathbb{F} $-linear map
$$ E_\Omega^S : \mathbb{F} [\mathbf{x} ; \sigma,\delta] \longrightarrow \mathbb{F}^\Omega, $$
where $ f = E_\Omega^S(F) \in \mathbb{F}^\Omega $ is given by $ f(\mathbf{a}) = F(\mathbf{a}) $, for all $ \mathbf{a} \in \Omega $ and all $ F \in \mathbb{F} [\mathbf{x} ; \sigma,\delta] $.
\end{definition}

Note that the skew evaluation map depends on the pair $ (\sigma, \delta) $. This will be the case with most of the objects defined from now on. However, we will not write such a dependency for brevity, unless it is necessary to avoid confusions.

We now turn to linearized (multivariate skew) polynomials. The idea is to turn skew polynomials into linear maps by giving an alternative evaluation map. Observe that, whereas skew polynomials are evaluated on points in $ \mathbb{F}^n $, linearized skew polynomials are evaluated on elements in $ \mathbb{F} $.

\begin{definition} [\textbf{Linearized multivariate skew polynomials}] \label{def lin multi skew pols}
Given a ring morphism $ \sigma : \mathbb{F} \longrightarrow \mathbb{F}^{n \times n} $, a $ \sigma $-derivation $ \delta : \mathbb{F} \longrightarrow \mathbb{F}^n $, a point $ \mathbf{a} \in \mathbb{F}^n $ and a monomial $ \mathfrak{m} \in \mathcal{M} $, we define the maps 
\begin{equation*}
\mathcal{D}_\mathbf{a}^\mathfrak{m} : \mathbb{F} \longrightarrow \mathbb{F}, 
\end{equation*}
recursively on $ \mathfrak{m} \in \mathcal{M} $, as follows. We start by defining $ \mathcal{D}_\mathbf{a}^1 = {\rm Id} $. Next, if $ \mathcal{D}_\mathbf{a}^\mathfrak{m} $ is defined for $ \mathfrak{m} \in \mathcal{M} $, then we define
$$ \mathcal{D}_\mathbf{a}^{\mathbf{x}\mathfrak{m}}(\beta) = \left( \begin{array}{c}
\mathcal{D}_\mathbf{a}^{x_1 \mathfrak{m}}(\beta) \\
\mathcal{D}_\mathbf{a}^{x_2 \mathfrak{m}}(\beta) \\
\vdots \\
\mathcal{D}_\mathbf{a}^{x_n \mathfrak{m}}(\beta)
\end{array} \right) = \sigma(\mathcal{D}_\mathbf{a}^\mathfrak{m}(\beta)) \mathbf{a} + \delta(\mathcal{D}_\mathbf{a}^\mathfrak{m}(\beta)) \in \mathbb{F}^n, $$
for all $ \beta \in \mathbb{F} $. For convenience, we denote by $ \mathcal{D}_\mathbf{a} : \mathbb{F} \longrightarrow \mathbb{F}^n $ the map given by $ \mathcal{D}_\mathbf{a}(\beta) = \sigma(\beta) \mathbf{a} + \delta(\beta) \in \mathbb{F}^n $, for all $ \beta \in \mathbb{F} $. Hence, by definition, we have that
\begin{equation}
\mathcal{D}^{\mathbf{x} \mathfrak{m}}_\mathbf{a} = \mathcal{D}_\mathbf{a} \circ \mathcal{D}^\mathfrak{m}_\mathbf{a},
\label{eq def of mathcal D map}
\end{equation}
for all $ \mathfrak{m} \in \mathcal{M} $ and all $ \mathbf{a} \in \mathbb{F} $. We then define the left vector space of \textit{linearized} (\textit{multivariate skew}) \textit{polynomials} $ \mathbb{F}[\mathcal{D}_\mathbf{a}] $ over $ \mathbb{F} $, with variables $ x_1, x_2, \ldots, x_n $, morphism $ \sigma $, derivation $ \delta $ and conjugacy representative $ \mathbf{a} $, as the left vector space generated by the set of maps $ \mathcal{D}_\mathbf{a}^\mathcal{M} = \{ \mathcal{D}_\mathbf{a}^\mathfrak{m} \mid \mathfrak{m} \in \mathcal{M} \} $. Note that $ \mathcal{D}_\mathbf{a}^\mathcal{M} $ need not be a basis nor be in bijection with $ \mathcal{M} $. We define the left $ \mathbb{F} $-linear (surjective) map
\begin{equation}
\begin{array}{rccc}
\phi_\mathbf{a} : & \mathbb{F}[\mathbf{x}; \sigma,\delta] & \longrightarrow & \mathbb{F}[\mathcal{D}_\mathbf{a}] \\
 & \sum_{\mathfrak{m} \in \mathcal{M}} F_\mathfrak{m} \mathfrak{m} & \mapsto & \sum_{\mathfrak{m} \in \mathcal{M}} F_\mathfrak{m} \mathcal{D}_\mathbf{a}^\mathfrak{m},
\end{array}
\label{eq map phi a}
\end{equation}
and we denote $ F^{\mathcal{D}_\mathbf{a}} = \phi_\mathbf{a}(F) $, for all $ F \in \mathbb{F}[\mathbf{x}; \sigma,\delta] $.
\end{definition}

If $ n = 1 $, $ \delta = 0 $ and $ \mathbf{a} = 1 $, and $ \mathbb{F} $ is a finite field, then $ \mathbb{F}[\mathcal{D}_\mathbf{a}] $ is the ring of classical (univariate) linearized polynomials over finite fields \cite[Ch. 3]{lidl} as considered originally by Ore \cite{orespecial}. In general, if $ \mathcal{G} $ is a finite group of ring automorphisms of $ \mathbb{F} $ generated by $ \sigma_1, \sigma_2, \ldots, \sigma_n $, and $ \sigma = {\rm diag}(\sigma_1, \sigma_2, $ $ \ldots, $ $ \sigma_n) $, $ \delta = 0 $ and $ \mathbf{a} = \mathbf{1} $, then $ \mathbb{F}[\mathcal{D}_\mathbf{a}] $ is the group ring of $ \mathcal{G} $ over $ \mathbb{F} $ \cite[p. 104]{lang} (note that if, in addition, $ \mathbb{F} $ is a finite field, then $ \mathbb{F}[\mathcal{D}_\mathbf{a}] $ is also a subset of the ring of conventional polynomials in one variable $ \mathbb{F}[t] $, as in the univariate case $ n = 1 $). Finally, if $ \sigma = {\rm Id} $ and $ \mathbf{a} = \mathbf{0} $, then $ \mathbb{F}[\mathcal{D}_\mathbf{a}] $ is the differential polynomial ring on the derivations $ \delta_1, \delta_2, \ldots, \delta_n $ over $ \mathbb{F} $ \cite[App. D]{singer}. These particular rings are recovered thanks to linearized evaluations as in Definition \ref{def lin evaluation}, whereas they are not recovered if we only consider remainder-based evaluations (Definition \ref{def skew evaluation}).

By definition, linearized polynomials admit a natural evaluation, as follows.

\begin{definition} [\textbf{Linearized evaluation}] \label{def lin evaluation}
Given $ \mathbf{a} \in \mathbb{F}^n $, we define the evaluation of $ F^{\mathcal{D}_\mathbf{a}} = \sum_{\mathfrak{m} \in \mathcal{M}} F_\mathfrak{m} \mathcal{D}_\mathbf{a}^\mathfrak{m} \in \mathbb{F}[\mathcal{D}_\mathbf{a}] $ in $ \beta \in \mathbb{F} $ as
$$ F^{\mathcal{D}_\mathbf{a}} (\beta) = \sum_{\mathfrak{m} \in \mathcal{M}} F_\mathfrak{m} \mathcal{D}_\mathbf{a}^\mathfrak{m}(\beta) \in \mathbb{F} . $$
Given $ \Omega \subseteq \mathbb{F} $, we define the linearized evaluation map over $ \Omega $ as the left $ \mathbb{F} $-linear map
$$ E^L_\Omega : \mathbb{F}[\mathcal{D}_\mathbf{a}] \longrightarrow \mathbb{F}^\Omega, $$
where $ f = E^L_\Omega(F^{\mathcal{D}_\mathbf{a}}) \in \mathbb{F}^\Omega $ is given by $ f(\beta) = F^{\mathcal{D}_\mathbf{a}}(\beta) $, for $ \beta \in \Omega $ and $ F^{\mathcal{D}_\mathbf{a}} \in \mathbb{F}[\mathcal{D}_\mathbf{a}] $.
\end{definition}

Linearized polynomials are right linear over certain division subrings of $ \mathbb{F} $, called centralizers, which were defined in \cite[Eq. (3.1)]{lam-leroy} in the univariate case.

\begin{definition} [\textbf{Centralizers}] \label{def centralizer}
Given $ \mathbf{a} \in \mathbb{F}^n $, we define its \textit{centralizer} as
$$ K_\mathbf{a} = \{ \beta \in \mathbb{F} \mid \mathcal{D}_\mathbf{a}(\beta) = \mathbf{a} \beta \} \subseteq \mathbb{F} . $$
\end{definition} 

\begin{lemma} \label{lemma lin pols are right linear}
For all $ \mathbf{a} \in \mathbb{F}^n $, it holds that $ K_\mathbf{a} \subseteq \mathbb{F} $ is a division subring of $ \mathbb{F} $. Moreover, for $ F^{\mathcal{D}_\mathbf{a}} \in \mathbb{F}[\mathcal{D}_\mathbf{a}] $, the map $ \beta \mapsto F^{\mathcal{D}_\mathbf{a}} (\beta) $, for $ \beta \in \mathbb{F} $, is right linear over $ K_\mathbf{a} $.
\end{lemma}

Thus, we have provided a type of evaluation that turns skew polynomials into linear maps over certain division subring. In fact, it is easy to show that centralizers are the largest division subrings over which linearized polynomials are right linear. 

\begin{proposition} \label{prop centralizer is largest for linearity}
Given $ \mathbf{a} \in \mathbb{F}^n $, $ K_\mathbf{a} $ is the largest division subring of $ \mathbb{F} $ over which every linearized polynomial in $ \mathbb{F}[\mathcal{D}_\mathbf{a}] $ is right linear, since
$$ K_\mathbf{a} = \left\lbrace \lambda \in \mathbb{F} \mid F^{\mathcal{D}_\mathbf{a}}(\beta \lambda) = F^{\mathcal{D}_\mathbf{a}}(\beta) \lambda, \textrm{ for } \beta \in \mathbb{F} \textrm{ and } F^{\mathcal{D}_\mathbf{a}} \in \mathbb{F}[\mathcal{D}_\mathbf{a}] \right\rbrace . $$
\end{proposition}

We now connect both types of evaluations (i.e., Definitions \ref{def skew evaluation} and \ref{def lin evaluation}). To that end, we use the so-called truncated norms and conjugacy. The following result is \cite[Th. 2]{multivariateskew}. In the univariate case $ n = 1 $, it shows that remainder-based evaluations of skew polynomials also generalize Abhyankar's \textit{projective polynomials} \cite{projective-pols}.

\begin{lemma}[\textbf{Truncated norms \cite{multivariateskew}}] \label{lemma multivariate norms}
Denote by $ N_\mathfrak{m}(\mathbf{a}) = E^S_\mathbf{a}(\mathfrak{m}) \in \mathbb{F} $ the evaluation of $ \mathfrak{m} \in \mathcal{M} $ at $ \mathbf{a} \in \mathbb{F}^n $ as in Definition \ref{def skew evaluation}. Then 
\begin{equation}
N_{\mathbf{x}\mathfrak{m}}(\mathbf{a}) = \left( \begin{array}{c}
N_{x_1 \mathfrak{m}}(\mathbf{a}) \\
N_{x_2 \mathfrak{m}}(\mathbf{a}) \\
\vdots \\
N_{x_n \mathfrak{m}}(\mathbf{a}) \\
\end{array} \right) = \mathcal{D}_\mathbf{a} ( N_\mathfrak{m}(\mathbf{a}) ).
\label{eq def of multivariate norm}
\end{equation}
\end{lemma}
%

We next revisit the concept of conjugacy \cite[Def. 11]{multivariateskew}.

\begin{definition}[\textbf{Conjugacy \cite{multivariateskew}}] \label{def conjugacy}
Given $ \mathbf{a} \in \mathbb{F}^n $ and $ \beta \in \mathbb{F}^* $, we define the conjugate of $ \mathbf{a} $ with respect to $ \beta $ as
\begin{equation}
\mathbf{a}^\beta = \mathcal{D}_\mathbf{a}(\beta)\beta^{-1} \in \mathbb{F}^n .
\label{eq def conjugate}
\end{equation}
\end{definition}

By \cite[Lemma 12]{multivariateskew}, conjugacy defines an equivalence relation in $ \mathbb{F}^n $, thus a partition of $ \mathbb{F}^n $ into conjugacy classes, which will be denoted by 
\begin{equation}
\mathcal{C} (\mathbf{a}) = \{ \mathcal{D}_\mathbf{a}(\beta)\beta^{-1} \mid \beta \in \mathbb{F}^* \} \subseteq \mathbb{F}^n,
\label{eq def conjugacy class}
\end{equation}
for $ \mathbf{a} \in \mathbb{F}^n $. We have the following lemma.

\begin{lemma} \label{lemma linear map of the conjugate}
Let $ \mathbf{a}, \mathbf{b} \in \mathbb{F}^n $ and $ \gamma \in \mathbb{F}^* $ be such that $ \mathbf{b} = \mathbf{a}^\gamma $. Then it holds that
$$ \mathcal{D}_\mathbf{b}(\beta) \gamma = \mathcal{D}_\mathbf{a}(\beta \gamma) \quad \textrm{and} \quad K_\mathbf{b} = \gamma K_\mathbf{a} \gamma^{-1}, $$
for all $ \beta \in \mathbb{F} $. In particular, if $ \mathbb{F} $ is commutative, then $ K_\mathbf{b} = K_\mathbf{a} $. 
\end{lemma}

We may now prove the connection between linearized and skew evaluations. This result extends \cite[Lemma 24]{linearizedRS} from the univariate to the multivariate case.

\begin{theorem} \label{theorem evaluation of skew as map}
Given $ \mathbf{a} \in \mathbb{F}^n $, $ \beta \in \mathbb{F}^* $ and $ F \in \mathbb{F}[\mathbf{x};\sigma,\delta] $, it holds that
$$ F(\mathcal{D}_\mathbf{a}(\beta)\beta^{-1}) = F^{\mathcal{D}_\mathbf{a}}(\beta)\beta^{-1}. $$
\end{theorem}
\begin{proof}
By linearity, we only need to prove that $ N_\mathfrak{m}(\mathcal{D}_\mathbf{a}(\beta)\beta^{-1}) = \mathcal{D}_\mathbf{a}^\mathfrak{m}(\beta)\beta^{-1} $ recursively on $ \mathfrak{m} \in \mathcal{M} $. The case $ \mathfrak{m} = 1 $ is trivial. Assume now that it is true for $ \mathfrak{m} \in \mathcal{M} $. Combining (\ref{eq def of mathcal D map}) and (\ref{eq def of multivariate norm}) with Lemma \ref{lemma linear map of the conjugate}, and denoting $ \mathbf{b} = \mathbf{a}^\beta = \mathcal{D}_\mathbf{a}(\beta)\beta^{-1} $, we conclude by
$$ N_{\mathbf{x}\mathfrak{m}}(\mathbf{b}) = \mathcal{D}_\mathbf{b}(N_\mathfrak{m}(\mathbf{b})) = \mathcal{D}_\mathbf{a}(N_\mathfrak{m}(\mathbf{b})\beta)\beta^{-1} = \mathcal{D}_\mathbf{a}(\mathcal{D}_\mathbf{a}^\mathfrak{m}(\beta)) \beta^{-1} = \mathcal{D}_\mathbf{a}^{\mathbf{x}\mathfrak{m}}(\beta)\beta^{-1}. $$
\end{proof}

\section{Linearizing sets of roots and P-independence} \label{sec linearized P-closed sets}

The structure of sets of roots play a central role in the study of conventional polynomials. In particular, Lagrange interpolation behaves well when the evaluation points can be differentiated by taking ``independent'' values on different polynomials. This is also true for skew polynomials and leads to the concepts of P-closed sets, P-independence and P-bases, where P stands for polynomial. Such concepts were introduced by Lam and Leroy in \cite{lam, algebraic-conjugacy, lam-leroy} in the univariate case, and in \cite{multivariateskew} for the multivariate case.

In this section, we extend the important results \cite[Th. 4.5]{lam-leroy} and \cite[Th. 23]{lam}, by Lam and Leroy, to the multivariate case. This provides an explicit characterization of P-closed sets as lists of vector spaces (Theorems \ref{theorem linearized version of P-closed in one conj} and \ref{theorem lin P-closed sets in several conj classes}), and a characterization of P-independence as linear independence per conjugacy class (Lemmas \ref{lemma linearized independence one conj class} and \ref{lemma P-independent several conj class}), showing that P-independent sets form a representable matroid \cite{oxley}.

Given a set $ \mathcal{A} \subseteq \mathbb{F}[\mathbf{x}; \sigma, \delta] $, we define its set of roots as
$$ Z( \mathcal{A} ) = \{ \mathbf{a} \in \mathbb{F}^n \mid F(\mathbf{a}) = 0, \textrm{ for all } F \in \mathcal{A} \}. $$
Conversely, given a set $ \Omega \subseteq \mathbb{F}^n $, we define its associated left ideal as
$$ I(\Omega) = \{ F \in \mathbb{F}[\mathbf{x}; \sigma, \delta] \mid F(\mathbf{a}) = 0, \textrm{ for all } \mathbf{a} \in \Omega \}. $$

As in \cite[Def. 16]{multivariateskew}, we define the \textit{P-closure} of a set $ \Omega \subseteq \mathbb{F}^n $ as $ \overline{\Omega} = Z(I(\Omega)) $, and we say that $ \Omega $ is \textit{P-closed} if $ \overline{\Omega} = \Omega $. As in \cite[Def. 22]{multivariateskew}, given a P-closed set $ \Omega \subseteq \mathbb{F}^n $, we say that $ \mathcal{G} \subseteq \Omega $ generates $ \Omega $ if $ \overline{\mathcal{G}} = \Omega $, and it is then called a set of \textit{P-generators} for $ \Omega $. We say that $ \Omega $ is finitely generated if it has a finite set of P-generators. As in \cite[Def. 23]{multivariateskew}, we say that $ \mathbf{a} \in \mathbb{F}^n $ is \textit{P-independent} from $ \Omega \subseteq \mathbb{F}^n $ if it does not belong to $ \overline{\Omega} $. A set $ \Omega \subseteq \mathbb{F}^n $ is called \textit{P-independent} if every $ \mathbf{a} \in \Omega $ is P-independent from $ \Omega \setminus \{ \mathbf{a} \} $. Finally, as in \cite[Def. 24]{multivariateskew}, we say that a subset $ \mathcal{B} \subseteq \Omega $ is a \textit{P-basis} of a P-closed set $ \Omega \subseteq \mathbb{F}^n $ if it is P-independent and a set of P-generators of $ \Omega $. 

If a P-closed set is finitely generated, then it admits a finite P-basis \cite[Cor. 26]{multivariateskew}, and any two of its P-bases are finite and have the same number of elements \cite[Cor. 32]{multivariateskew}. Thus, we may define the \textit{rank} of a finitely generated P-closed set $ \Omega \subseteq \mathbb{F}^n $, denoted by $ {\rm Rk}(\Omega) $, as the size of any of its P-bases. In fact, the collection of P-independent subsets of a finitely generated P-closed set forms a matroid \cite[Prop. 27]{multivariateskew} and the concepts above correspond to those in classical matroid theory \cite{oxley}.

The main feature of P-bases of finitely generated P-closed sets is the following result on Lagrange interpolating skew polynomials, which is \cite[Th. 4]{multivariateskew}.

\begin{theorem}[\textbf{Skew Lagrange interpolation \cite{multivariateskew}}] \label{th lagrange interpolation}
Let $ \Omega \subseteq \mathbb{F}^n $ be a P-closed set with P-basis $ \mathcal{B} = \{ \mathbf{b}_1, \mathbf{b}_2, \ldots, \mathbf{b}_M \} $, where $ M = {\rm Rk}(\Omega) < \infty $. The following hold:
\begin{enumerate}
\item
If $ E^S_{\mathcal{B}}(F) = E^S_{\mathcal{B}}(G) $, then $ E^S_{\Omega}(F) = E^S_{\Omega}(G) $, for all $ F,G \in \mathbb{F}[\mathbf{x}; \sigma, \delta] $. 
\item
For every $ a_1, a_2, \ldots, a_M \in \mathbb{F} $, there exists $ F \in \mathbb{F}[\mathbf{x}; \sigma, \delta] $ such that $ \deg(F) < M $ and $ F(\mathbf{b}_i) = a_i $, for $ i = 1,2, \ldots, M $.
\end{enumerate}
\end{theorem}

Skew Lagrange interpolation leads to the concept of dual P-bases \cite[Def. 30]{multivariateskew}, which we will use to extend \cite[Th. 4.5]{lam-leroy} and \cite[Th. 23]{lam} to the multivariate case.

\begin{definition} [\textbf{Dual P-bases \cite{multivariateskew}}] \label{def dual P-bases}
Given a P-basis $ \mathcal{B} = \{ \mathbf{b}_1, \mathbf{b}_2, \ldots, \mathbf{b}_M \} $ of a P-closed set $ \Omega \subseteq \mathbb{F}^n $, a dual P-basis of $ \mathcal{B} $ is a set $ \mathcal{B}^* = \{ F_1, F_2, \ldots, F_M \} \subseteq \mathbb{F}[\mathbf{x}; \sigma, \delta] $ such that $ F_i(\mathbf{b}_j) = \delta_{i,j} $ (the Kronecker delta), for all $ i,j = 1,2, \ldots, M $.
\end{definition}

By Theorem \ref{th lagrange interpolation}, any P-basis of a P-closed set of rank $ M < \infty $ admits a dual P-basis consisting of $ M $ skew polynomials of degree less than $ M $. 

We will also use the following two technical tools throughout this section. 

\begin{corollary} [\textbf{\cite{multivariateskew}}] \label{cor inverse projection map}
If $ \{ F_1, F_2, \ldots, F_M \} $ is a dual P-basis of a finitely generated P-closed set $ \Omega \subseteq \mathbb{F}^n $, the projection map restricts to a left vector space isomorphism
$$ \langle F_1, F_2, \ldots, F_M \rangle^L_\mathbb{F} \cong \mathbb{F}[\mathbf{x}; \sigma, \delta] / I(\Omega) $$
over $ \mathbb{F} $. Moreover, $ F_1, F_2, \ldots, F_M  $ are left linearly independent over $ \mathbb{F} $, hence
$$ \dim^L_\mathbb{F} \left( \mathbb{F}[\mathbf{x}; \sigma, \delta] / I(\Omega) \right) = {\rm Rk}(\Omega). $$
\end{corollary}

The following product rule was given in \cite[Th. 2.7]{lam-leroy} in the univariate case, and in \cite[Th. 3]{multivariateskew} in the multivariate case.

\begin{theorem}[\textbf{Product rule \cite{multivariateskew}}] \label{th product rule}
Given skew polynomials $ F, G \in \mathbb{F}[\mathbf{x}; \sigma, \delta] $ and $ \mathbf{a} \in \mathbb{F}^n $, if $ G(\mathbf{a}) = 0 $, then $ (FG)(\mathbf{a}) = 0 $, and if $ \beta = G(\mathbf{a}) \neq 0 $, then
\begin{equation*}
(FG)(\mathbf{a}) = F(\mathbf{a}^\beta) G(\mathbf{a}).
\end{equation*}
\end{theorem}

We may now characterize P-independence in one conjugacy class as right linear independence over the corresponding centralizer.

\begin{lemma} \label{lemma linearized independence one conj class}
Let $ \mathbf{a}, \mathbf{b}_1, \mathbf{b}_2, \ldots, \mathbf{b}_M \in \mathbb{F}^n $ and $ \beta_1, \beta_2, \ldots, \beta_M \in \mathbb{F}^* $ be such that 
$$ \mathbf{b}_i = \mathbf{a}^{\beta_i} = \mathcal{D}_\mathbf{a}(\beta_i) \beta_i^{-1}, $$
for $ i = 1,2, \ldots, M $. Then $ \mathcal{B} = \{ \mathbf{b}_1, \mathbf{b}_2, \ldots, \mathbf{b}_M \} $ is P-independent if, and only if, $ \mathcal{B}^\mathcal{D} = \{ \beta_1, \beta_2, \ldots, \beta_M \} $ is right linearly independent over $ K_\mathbf{a} $.
\end{lemma}
\begin{proof}
We first prove the direct implication. Assume that $ \mathcal{B} $ is P-independent, but $ \mathcal{B}^\mathcal{D} $ is not right linearly independent over $ K_\mathbf{a} $. We may assume without loss of generality that there exist $ \lambda_1, \lambda_2, \ldots, \lambda_{M-1} \in K_\mathbf{a} $ such that
$$ \beta_M = \sum_{i=1}^{M-1} \beta_i \lambda_i. $$
By Theorem \ref{th lagrange interpolation}, there exists $ F \in \mathbb{F}[\mathbf{x}, \sigma, \delta] $ such that $ F(\beta_M) = 1 $ and $ F(\beta_i) = 0 $, for $ i = 1,2, \ldots, M-1 $. Therefore by Lemma \ref{lemma lin pols are right linear} and Theorem \ref{theorem evaluation of skew as map}, it holds that
$$ \beta_M = F_M^{\mathcal{D}_\mathbf{a}}(\beta_M) = \sum_{i=1}^{M-1} F_M^{\mathcal{D}_\mathbf{a}}(\beta_i) \lambda_i = 0, $$
which is absurd since $ \beta_M \in \mathbb{F}^* $ by hypothesis.

Conversely, assume that $ \mathcal{B}^\mathcal{D} $ is right linearly independent over $ K_\mathbf{a} $. We will prove by induction on $ M $ that $ \mathcal{B} $ is P-independent. The case $ M = 1 $ is obvious since all singleton sets are P-independent. Assume that $ \mathcal{B}^\prime = \{ \mathbf{b}_1, \mathbf{b}_2, \ldots, \mathbf{b}_{M-1} \} $ is P-independent but $ \mathcal{B} $ is not. Then $ \mathbf{b}_M \in \overline{\mathcal{B}^\prime} $, since otherwise $ \mathcal{B} $ would be P-independent \cite[Lemma 36]{multivariateskew}.

Let $ \mathcal{B}^{\prime *} = \{ F_1, F_2, \ldots, F_{M-1} \} $ be a dual P-basis of $ \mathcal{B}^\prime $. Fix $ i = 1,2, \ldots, M-1 $ and define $ \mathbf{G}_i = (\mathbf{x} - \mathbf{b}_i) F_i \in \mathbb{F}[\mathbf{x}; \sigma, \delta]^n $. It holds that $ \mathbf{G}_i \in I(\mathcal{B}^\prime) $ by the product rule (Theorem \ref{th product rule}). Since $ \mathbf{b}_M \in \overline{\mathcal{B}^\prime} = Z(I(\mathcal{B}^\prime)) $, then $ \mathbf{G}_i (\mathbf{b}_M) = \mathbf{0} $. If $ F_i(\mathbf{b}_M) \neq 0 $, then
$$ \mathbf{0} = \mathbf{G}_i (\mathbf{b}_M) = \left( \mathbf{b}_M^{F_i(\mathbf{b}_M)} - \mathbf{b}_i \right) F_i(\mathbf{b}_M) = \left( \mathbf{a}^{F_i(\mathbf{b}_M) \beta_M} - \mathbf{a}^{\beta_i} \right)F_i(\mathbf{b}_M) , $$
again by the product rule (Theorem \ref{th product rule}). Now, $ \mathbf{a}^{F_i(\mathbf{b}_M) \beta_M} = \mathbf{a}^{\beta_i} $ means that $ \beta_i^{-1} F_i(\mathbf{b}_M) \beta_M \in K_\mathbf{a} $ by Definition \ref{def centralizer}. Hence in all cases ($ F_i(\mathbf{b}_M) = 0 $ or $ F_i(\mathbf{b}_M) \neq 0 $) we have that
$$ F_i(\mathbf{b}_M) = \beta_i \lambda_i \beta_M^{-1}, $$
for some $ \lambda_i \in K_\mathbf{a} $. Now, if $ F = F_1 + F_2 + \cdots + F_{M-1} \in \mathbb{F}[\mathbf{x}; \sigma, \delta] $, then
$$ F(\mathbf{b}_j) = \sum_{i=1}^{M-1} F_i(\mathbf{b}_j) = \sum_{i=1}^{M-1} \delta_{i,j} = 1, $$
for $ j = 1,2, \ldots, M-1 $, by Definition \ref{def dual P-bases}. Since $ \mathbf{b}_M \in \overline{\mathcal{B}^\prime} = Z(I(\mathcal{B}^\prime)) $, we deduce from Theorem \ref{th lagrange interpolation} that $ F(\mathbf{b}_M) = 1 $. Hence
$$ 1 = F(\mathbf{b}_M) = \sum_{i=1}^{M-1} F_i(\mathbf{b}_M) = \sum_{i=1}^{M-1} \beta_i \lambda_i \beta_M^{-1} \quad \Longleftrightarrow \quad \beta_M = \sum_{i=1}^{M-1} \beta_i \lambda_i, $$
which contradicts the right linear independence of $ \beta_1, \beta_2, \ldots, \beta_{M-1} $ over $ K_\mathbf{a} $.
\end{proof}

The second important ingredient is to ensure that P-closed sets generated by a finite set inside a single conjugacy class remain contained in such a conjugacy class. 

\begin{lemma} \label{lemma P-closed are conj closed}
If $ \mathcal{G} \subseteq \mathbb{F}^n $ is finite and $ \mathbf{b} \in \overline{\mathcal{G}} $, then $ \mathbf{b} $ is conjugate to an element in $ \mathcal{G} $.
\end{lemma}
\begin{proof}
Let $ \mathcal{B} = \{ \mathbf{b}_1, \mathbf{b}_2, \ldots, \mathbf{b}_M \} \subseteq \mathcal{G} $ be a P-basis of $ \overline{\mathcal{G}} $ and let $ \mathcal{B}^* = \{ F_1, F_2, \ldots, F_M \} $ be a dual P-basis of $ \mathcal{B} $. There exists $ i = 1,2, \ldots, M $ such that $ F_i(\mathbf{b}) \neq 0 $, since otherwise we deduce from Corollary \ref{cor inverse projection map} that $ F(\mathbf{b}) = 0 $, for all $ F \in \mathbb{F}[\mathbf{x}; \sigma, \delta] $, which is absurd. However, if $ \mathbf{G}_i = (\mathbf{x} - \mathbf{b}_i) F_i \in \mathbb{F}[\mathbf{x}; \sigma, \delta]^n $, then $ \mathbf{G}_i \in I(\mathcal{B}) $ by the product rule (Theorem \ref{th product rule}). Since $ \mathbf{b} \in \overline{\mathcal{B}} = Z(I(\mathcal{B})) $, we deduce, again from the product rule, that
$$ \mathbf{0} = \mathbf{G}_i(\mathbf{b}) = \left( \mathbf{b}^{F_i(\mathbf{b})} - \mathbf{b}_i \right) F_i(\mathbf{b}). $$
Therefore $ \mathbf{b} $ is conjugate to $ \mathbf{b}_i \in \mathcal{G} $ and we are done.
\end{proof}

We may now give the first main result of this section, which characterizes finitely generated P-closed subsets of a conjugacy class as vector spaces.

\begin{theorem} \label{theorem linearized version of P-closed in one conj}
Let $ \mathbf{a} \in \mathbb{F}^n $. The following hold:
\begin{enumerate}
\item
If $ \mathcal{G} \subseteq \mathcal{C}(\mathbf{a}) $ is finite and $ \Omega = \overline{\mathcal{G}} \subseteq \mathbb{F}^n $, then
\begin{equation}
\Omega = \{ \mathbf{a}^\beta \mid \beta \in \Omega^{\mathcal{D}} \setminus \{ 0 \} \} \subseteq \mathcal{C}(\mathbf{a}),
\label{eq omega is linearized one conj}
\end{equation}
for a finite-dimensional right vector space $ \Omega^\mathcal{D} \subseteq \mathbb{F} $ over $ K_\mathbf{a} $.
\item
Conversely, if $ \Omega^\mathcal{D} \subseteq \mathbb{F} $ is a finite-dimensional right vector space over $ K_\mathbf{a} $, then $ \Omega \subseteq \mathcal{C}(\mathbf{a}) $ given as in (\ref{eq omega is linearized one conj}) is a finitely generated P-closed set.
\end{enumerate}
Moreover if Item 1 or 2 holds, then $ \mathcal{B}^\mathcal{D} $ is a right basis of $ \Omega^\mathcal{D} $ over $ K_\mathbf{a} $ if, and only if, $ \mathcal{B} = \{ \mathbf{a}^\beta \in \mathbb{F}^* \mid \beta \in \mathcal{B}^\mathcal{D} \} $ is a P-basis of $ \Omega $. In particular, we have that
\begin{equation}
{\rm Rk}(\Omega) = \dim^R_{K_\mathbf{a}} (\Omega^\mathcal{D}).
\label{eq rank is dimension of linearized}
\end{equation}
Thus we deduce that the map $ \Omega \mapsto \Omega^\mathcal{D} $ is a bijection between finitely generated P-closed subsets of $ \mathcal{C}(\mathbf{a}) $ and finite-dimensional righ vector subspaces of $ \mathbb{F} $ over $ K_\mathbf{a} $. 
\end{theorem}
\begin{proof}
Assume first the hypotheses in Item 1, and let $ \mathcal{B} = \{ \mathbf{b}_1, \mathbf{b}_2, \ldots, \mathbf{b}_M \} \subseteq \Omega $ be a P-basis of $ \Omega $. By Lemma \ref{lemma P-closed are conj closed}, we have that $ \Omega \subseteq \mathcal{C}(\mathbf{a}) $. Hence there exist $ \beta_i \in \mathbb{F}^* $ such that $ \mathbf{b}_i = \mathbf{a}^{\beta_i} $, for $ i = 1,2, \ldots, M $. Therefore, Equation (\ref{eq omega is linearized one conj}) holds for $ \Omega^\mathcal{D} = \langle \beta_1, \beta_2, \ldots, \beta_M \rangle_{K_\mathbf{a}}^R \subseteq \mathbb{F} $ by Lemma \ref{lemma linearized independence one conj class}. Thus Item 1 is proven.

Assume now the hypotheses in Item 2. Let $ \mathcal{B}^\mathcal{D} = \{ \beta_1, \beta_2, \ldots, \beta_M \} $ be a right basis of $ \Omega^\mathcal{D} $ over $ K_\mathbf{a} $, and define $ \mathbf{b}_i = \mathbf{a}^{\beta_i} \in \Omega $, for $ i = 1,2, \ldots, M $. By Lemma \ref{lemma linearized independence one conj class}, it holds that $ \mathbf{a}^\beta \in \Omega $ if, and only if, $ \beta \in \Omega^\mathcal{D} $, for all $ \beta \in \mathbb{F}^* $. Thus Item 2 is proven.

Similarly, the claims under Items 1 and 2 follow from Lemma \ref{lemma linearized independence one conj class}, and we are done.
\end{proof}

We may deduce the following important consequence. As we will show in Section \ref{sec generalizations galois and derivations}, this consequence is a generalization of Hilbert's Theorem 90 \cite[Th. 21]{artin-lectures} \cite[Th. 90]{hilbert}. The proof is straightforward from Lemma \ref{lemma linearized independence one conj class} and Theorem \ref{theorem linearized version of P-closed in one conj}, and is left to the reader. 

\begin{corollary} \label{cor first version hilbert 90}
Let $ \mathbf{a} \in \mathbb{F} $. The conjugacy class $ \mathcal{C}(\mathbf{a}) \subseteq \mathbb{F}^n $ is P-closed and finitely generated if, and only if, $ \mathbb{F} $ has finite right dimension over $ K_\mathbf{a} $. In such a case, 
$$ {\rm Rk}(\mathcal{C}(\mathbf{a})) = \dim^R_{K_\mathbf{a}}(\mathbb{F}). $$
\end{corollary}

We may also deduce the following important consequence. As we will show in Section \ref{sec generalizations galois and derivations}, this result generalizes Artin's Theorem \cite[Th. 14]{artin-lectures}. 

\begin{corollary} \label{cor dimensions for artin}
For all $ \mathbf{a} \in \mathbb{F}^n $, the surjective map $ \phi_\mathbf{a} $ in (\ref{eq map phi a}) satisfies that $ {\rm Ker}(\phi_\mathbf{a}) = I(\mathcal{C}(\mathbf{a})) $ (by Theorem \ref{theorem evaluation of skew as map}), thus it restricts to a left $ \mathbb{F} $-linear vector space isomorphism
\begin{equation}
\phi_\mathbf{a} : \mathbb{F}[\mathbf{x}; \sigma, \delta]/I(\mathcal{C}(\mathbf{a})) \longrightarrow \mathbb{F}[\mathcal{D}_\mathbf{a}].
\label{eq isomorphism for artin}
\end{equation}
By Corollaries \ref{cor inverse projection map} and \ref{cor first version hilbert 90}, $ \mathbb{F}[\mathcal{D}_\mathbf{a}] $ has finite left dimension over $ \mathbb{F} $ if, and only if, $ \mathbb{F} $ has finite right dimension over $ K_\mathbf{a} $, and in such a case, 
$$ \dim^L_\mathbb{F}(\mathbb{F}[\mathcal{D}_\mathbf{a}]) = \dim^L_\mathbb{F}(\mathbb{F}[\mathbf{x}; \sigma, \delta]/I(\mathcal{C}(\mathbf{a}))) = {\rm Rk}(\mathcal{C}(\mathbf{a})) = \dim^R_{K_\mathbf{a}}(\mathbb{F}) . $$ 
\end{corollary}

We also deduce that the set of P-closed subsets of a conjugacy class forms a lattice that is isomorphic to the lattice of right projective subspaces of $ \mathbb{P}_{K_\mathbf{a}}(\mathbb{F}) $.

\begin{corollary} \label{cor lattice projective spaces}
Let $ \mathbf{a} \in \mathbb{F} $, and define the sum of two finitely generated P-closed sets $ \Omega_1, \Omega_2 \subseteq \mathcal{C}(\mathbf{a}) $ as $ \Omega_1 + \Omega_2 = \overline{\Omega_1 \cup \Omega_2} \subseteq \mathcal{C}(\mathbf{a}) $. The collection of finitely generated P-closed subsets of $ \mathcal{C}(\mathbf{a}) $ forms a lattice with sums and intersections isomorphic to the lattice of finite-dimensional right $ K_\mathbf{a} $-linear projective subspaces of $ \mathbb{P}^R_{K_\mathbf{a}}(\mathbb{F}) $ via the bijection
\begin{equation*}
\begin{array}{rccc}
\pi_\mathbf{a} : & \mathbb{P}^R_{K_\mathbf{a}}(\mathbb{F}) & \longrightarrow & \mathcal{C}(\mathbf{a}) \\
 & \left[ \beta \right] & \mapsto & \mathbf{a}^\beta,
\end{array}
\end{equation*}
where $ \left[ \beta \right] = \{ \beta \lambda \in \mathbb{F}^* \mid \lambda \in K_\mathbf{a}^* \} $ and $ \mathbb{P}^R_{K_\mathbf{a}}(\mathbb{F}) = \{ [\beta] \mid \beta \in \mathbb{F}^* \} $. 
\end{corollary}

We now explore what happens when we ``glue'' several conjugacy classes together. We start with the following lemma, which extends \cite[Th. 22]{lam}.

\begin{lemma} \label{lemma P-independent several conj class}
If $ \mathcal{B}_1, \mathcal{B}_2 \subseteq \mathbb{F}^n $ are finite non-empty P-independent sets such that no element in $ \mathcal{B}_1 $ is conjugate to an element in $ \mathcal{B}_2 $, then $ \mathcal{B} = \mathcal{B}_1 \cup \mathcal{B}_2 $ is P-independent.
\end{lemma}
\begin{proof}
Let $ \mathcal{B}_1 = \{ \mathbf{b}_1, \mathbf{b}_2, \ldots, \mathbf{b}_M \} $ and $ \mathcal{B}_2 = \{ \mathbf{c}_1, \mathbf{c}_2, \ldots, \mathbf{c}_N \} $, where $ M = |\mathcal{B}_1| \geq 1 $ and $ N = |\mathcal{B}_2| \geq 1 $. We proceed by induction on $ k = M + N $. The case $ k = 2 $ ($ M = N = 1 $) is trivial, since any set of two elements is P-independent. Assume now that the lemma holds for some $ k \geq 2 $, but not for $ k+1 $. By Lemma \ref{lemma P-closed are conj closed}, we may assume that $ N+1 = | \mathcal{B}_2| $ and $ \mathbf{c}_{N+1} \in \overline{\mathcal{B}} $, where $ \mathcal{B} = \{ \mathbf{b}_1, \mathbf{b}_2, \ldots, \mathbf{b}_M, \mathbf{c}_1, \mathbf{c}_2, \ldots, \mathbf{c}_N \} $.

By induction hypothesis, $ \mathcal{B} $ and $ \mathcal{B}_2 $ are P-independent, thus we may take dual P-bases $ \mathcal{B}^* = \{ F_1, F_2, \ldots, F_M, G_1, G_2, \ldots, G_N \} $ and $ \mathcal{B}_2^* = \{ H_1, H_2, \ldots, $ $ H_{N+1} \} $. 

First we prove that $ F_i(\mathbf{c}_{N+1}) = 0 $, for all $ i = 1,2, \ldots, M $. Assume that $ F_i(\mathbf{c}_{N+1}) \neq 0 $ for certain $ i $. It holds that $ \mathbf{G}_i = (\mathbf{x} - \mathbf{b}_i) F_i \in I(\mathcal{B})^n $, and since $ \mathbf{c}_{N+1} \in \overline{\mathcal{B}} $, then
$$ \mathbf{0} = \mathbf{G}_i(\mathbf{c}_{N+1}) = \left( \mathbf{c}_{N+1}^{F_i(\mathbf{c}_{N+1})} - \mathbf{b}_i \right) F_i(\mathbf{c}_{N+1}), $$
hence $ \mathbf{c}_{N+1} $ and $ \mathbf{b}_i $ are conjugate, which is a contradiction. Next define
$$ F = H_{N+1} - \sum_{i=1}^M H_{N+1}(\mathbf{b}_i) F_i. $$
It holds that $ F(\mathbf{b}_i) = F(\mathbf{c}_j) = 0 $, for all $ i = 1,2, \ldots, M $ and all $ j = 1,2, \ldots, N $. That is, $ F \in I(\mathcal{B}) $, and since $ \mathbf{c}_{N+1} \in \overline{\mathcal{B}} $, we have that $ F(\mathbf{c}_{N+1}) = 0 $. In other words,
$$ 0 = F(\mathbf{c}_{N+1}) = H_{N+1}(\mathbf{c}_{N+1}) - \sum_{i=1}^M H_{N+1}(\mathbf{b}_i) F_i(\mathbf{c}_{N+1}) = 1 - 0, $$
which is absurd, and we are done.
\end{proof}

We may now state and prove the second main result of this section.

\begin{theorem} \label{theorem lin P-closed sets in several conj classes}
If $ \Omega \subseteq \mathbb{F}^n $ is P-closed and finitely generated, then so is $ \Omega \cap \mathcal{C}(\mathbf{a}) $ for all $ \mathbf{a} \in \mathbb{F}^n $. Conversely, if the sets $ \Omega_i \subseteq \mathcal{C}(\mathbf{a}_i) $ are P-closed and finitely generated, for $ i = 1,2, \ldots, \ell $, where $ \mathbf{a}_1, \mathbf{a}_2, \ldots, \mathbf{a}_\ell \in \mathbb{F}^n $ are pair-wise non-conjugate, then $ \Omega = \Omega_1 \cup \Omega_2 \cup \ldots \cup \Omega_\ell $ is P-closed and finitely generated.

In addition, if $ \mathcal{B}_i $ is a P-basis of $ \Omega_i $, for $ i = 1,2, \ldots, \ell $, then $ \mathcal{B} = \mathcal{B}_1 \cup \mathcal{B}_2 \cup \ldots \cup \mathcal{B}_\ell $ is a P-basis of $ \Omega $, and in particular we have that
$$ {\rm Rk}(\Omega) = {\rm Rk}(\Omega_1) + {\rm Rk}(\Omega_2) + \cdots + {\rm Rk}(\Omega_\ell) . $$

By Corollary \ref{cor lattice projective spaces}, the collection of finitely generated P-closed subsets of $ \mathcal{C}(\mathbf{a}_1) \cup  \mathcal{C}(\mathbf{a}_2) \cup \ldots \cup \mathcal{C}(\mathbf{a}_\ell) $ forms a lattice isomorphic to the Cartesian product of the lattices of finite-dimensional right $ K_{\mathbf{a}_i} $-linear projective subspaces of $ \mathbb{P}^R_{K_{\mathbf{a}_i}}(\mathbb{F}) $, for $ i = 1,2, \ldots, \ell $.
\end{theorem}
\begin{proof}
Let $ \Omega $ be P-closed and finitely generated, and assume that $ \Omega \cap \mathcal{C}(\mathbf{a}) \neq \varnothing $. Let $ \mathcal{G} \subseteq \Omega \cap \mathcal{C}(\mathbf{a}) $ be a finite maximal P-independent set, which exists by \cite[Cor. 37]{multivariateskew}. By Lemma \ref{lemma P-closed are conj closed}, $ \overline{\mathcal{G}} \subseteq \Omega \cap \mathcal{C}(\mathbf{a}) $. By maximality of $ \mathcal{G} $ and by \cite[Lemma 36]{multivariateskew}, we conclude that $ \overline{\mathcal{G}} = \Omega \cap \mathcal{C}(\mathbf{a}) $, thus $ \Omega \cap \mathcal{C}(\mathbf{a}) $ is P-closed and finitely generated.

Now let $ \Omega_i \subseteq \mathcal{C}(\mathbf{a}_i) $ be P-closed sets with finite P-bases $ \mathcal{B}_i $, for $ i = 1,2, \ldots, \ell $, as in the theorem, and define $ \Omega = \Omega_1 \cup \Omega_2 \cup \ldots \cup \Omega_\ell $. First $ \mathcal{B} = \mathcal{B}_1 \cup \mathcal{B}_2 \cup \ldots \cup \mathcal{B}_\ell $ is P-independent by Lemma \ref{lemma P-independent several conj class}, hence we are done if we prove that $ \Omega = \overline{\mathcal{B}} $. First, $ \Omega = \overline{\mathcal{B}}_1 \cup \overline{\mathcal{B}}_2 \cup \ldots \cup \overline{\mathcal{B}}_\ell \subseteq \overline{\mathcal{B}} $. Hence, we only need to prove the reversed inclusion $ \overline{\mathcal{B}} \subseteq \Omega $.

Let $ \mathbf{b} \in \overline{\mathcal{B}} $. By Lemma \ref{lemma P-closed are conj closed}, there exists $ j = 1,2, \ldots, \ell $ such that $ \mathbf{b} $ is conjugate to an element in $ \mathcal{B}_j $. Define $ \mathcal{B}^\prime = \bigcup_{i \neq j} \mathcal{B}_i = \{ \mathbf{b}_1, \mathbf{b}_2, \ldots, \mathbf{b}_M \} $ and let $ \mathcal{B}^{\prime *} = \{ F_1, F_2, \ldots, F_M \} $ be one of its dual P-bases. Let $ F \in I(\mathcal{B}_j) $ and define
$$ G = F - \sum_{i=1}^M F(\mathbf{b}_i) F_i. $$
Since no element in $ \mathcal{B}_j \cup \{ \mathbf{b} \} $ is conjugate to an element in $ \mathcal{B}^\prime $, it must hold that $ F_i(\mathbf{c}) = 0 $ as in the proof of Lemma \ref{lemma P-closed are conj closed}, for all $ \mathbf{c} \in \mathcal{B}_j \cup \{ \mathbf{b} \} $ and all $ i = 1,2, \ldots, M $. In particular, $ G \in I(\mathcal{B}) $, and thus $ G(\mathbf{b}) = 0 $. Hence $ F(\mathbf{b}) = 0 $ and $ \mathbf{b} \in \overline{\mathcal{B}}_j = \Omega_j \subseteq \Omega $.
\end{proof}

\section{Skew and linearized polynomial arithmetic} \label{sec skew and lin polynomial arithmetic}

In this section, we study and relate the arithmetic of the rings of skew polynomials (Definition \ref{def free skew polynomials}) and linearized polynomials (Definition \ref{def lin multi skew pols}). We show that, on one conjugacy class, products of skew polynomials are mapped onto compositions of linearized polynomials (Theorem \ref{th skew pol is composition}) and matrix products (Theorem \ref{th composition is matrix multiplication}). Second, we show that, over several conjugacy classes, products of skew polynomials decompose into coordinate-wise compositions of linearized polynomials and products of matrices (Theorem \ref{th product decompositions}). Throughout this section, for $ \mathbf{a} \in \mathbb{F}^n $, we consider $ \mathbb{F}[\mathcal{D}_\mathbf{a}] $ as a ring with map composition $ \circ $. 

\begin{theorem} \label{th skew pol is composition}
For $ \mathbf{a} \in \mathbb{F} $ and $ F , G \in \mathbb{F}[\mathbf{x}; \sigma, \delta] $, it holds that
\begin{equation}
(FG)^{\mathcal{D}_\mathbf{a}} = F^{\mathcal{D}_\mathbf{a}} \circ G^{\mathcal{D}_\mathbf{a}} .
\label{eq skew pol product is composition}
\end{equation}
\end{theorem}
\begin{proof}
The theorem follows from the universal property of $ \mathbb{F}[\mathbf{x}; \sigma, \delta] $ (Lemma \ref{lemma universal property}), since we have that, for all $ \beta \in \mathbb{F} $ and all $ i = 1,2, \ldots, n $,
$$ \mathcal{D}_\mathbf{a}^{x_i} \circ (\beta {\rm Id}) = \sum_{j=1}^n \sigma_{i,j}(\beta) \mathcal{D}_\mathbf{a}^{x_j} + \delta_i(\beta) {\rm Id} . $$
\end{proof}

Now we turn to matrix multiplication. For simplicity, we will fix $ \mathbf{a} \in \mathbb{F}^n $ and assume that $ M = \dim^R_{K_\mathbf{a}}(\mathbb{F}) < \infty $. We will also fix a right ordered basis $ \boldsymbol\beta = ( \beta_{1}, \beta_{2}, \ldots, $ $ \beta_M ) \in \mathbb{F}^M $ of $ \mathbb{F} $ over $ K_\mathbf{a} $. Define the map $ \mu_{\boldsymbol\beta} : \mathbb{F}^M \longrightarrow K_\mathbf{a}^{M \times M} $ by 
\begin{equation}
\mu_{\boldsymbol\beta} \left(  \mathbf{x}  \right) = \left( \begin{array}{cccc}
x^1_1 & x^1_2 & \ldots & x^1_M \\
x^2_1 & x^2_2 & \ldots & x^2_M \\
\vdots & \vdots & \ddots & \vdots \\
x^M_1 & x^M_2 & \ldots & x^M_M \\
\end{array} \right),
\label{eq def matrix representation map}
\end{equation}
for $ \mathbf{x} = (x_1, x_2, \ldots, x_M) \in \mathbb{F}^M $, where $ x_j^1, x_j^2, \ldots, x_j^M \in K_\mathbf{a} $ are the unique scalars such that $ x_j = \sum_{i=1}^M \beta_i x_j^i \in \mathbb{F} $, for $ j = 1,2, \ldots, M $. Observe that $ \mu_{\boldsymbol\beta} $ is a right $ K_\mathbf{a} $-linear vector space isomorphism, and it is the identity map if $ M = 1 $ and $ \beta_1 = 1 $. 
	
Now, given $ \mathbf{x}, \mathbf{y} \in \mathbb{F}^M $, we define their \textit{matrix product} with respect to $ \boldsymbol\beta $ as
\begin{equation}
\mathbf{x} \star \mathbf{y} = \mu_{\boldsymbol\beta}^{-1}(\mu_{\boldsymbol\beta}(\mathbf{x}) \mu_{\boldsymbol\beta}(\mathbf{y})) \in \mathbb{F}^M ,
\label{eq def matrix-schur products}
\end{equation}
which depends on $ K_\mathbf{a} \subseteq \mathbb{F} $ and $ \boldsymbol\beta $. We will not denote this dependency for simplicity. 

From the definitions, we note also that, if $ \mathbf{x} = (x_{1}, x_{2}, \ldots, x_M) \in \mathbb{F}^M $ and $ \mathbf{y} = \sum_{i=1}^M \beta_{i} \mathbf{y}^i \in \mathbb{F}^M $, with $ x_{i} \in \mathbb{F} $ and $ \mathbf{y}^i \in K_\mathbf{a}^M $, for $ i = 1,2, \ldots, M $, then
\begin{equation}
\mu_{\boldsymbol\beta}^{-1}(\mu_{\boldsymbol\beta}(\mathbf{x}) \mu_{\boldsymbol\beta}(\mathbf{y})) = \sum_{i=1}^M x_i \mathbf{y}^i \in \mathbb{F}^M.
\label{eq rewriting matrix-Schur product}
\end{equation}

We may now prove the second main result of this section.

\begin{theorem} \label{th composition is matrix multiplication}
With notation and assumptions as above, for all $ F^{\mathcal{D}_\mathbf{a}}, G^{\mathcal{D}_\mathbf{a}} \in \mathbb{F}[\mathcal{D}_\mathbf{a}] $,
\begin{equation}
\left( F^{\mathcal{D}_\mathbf{a}} \circ G^{\mathcal{D}_\mathbf{a}} \right) (\boldsymbol\beta) = F^{\mathcal{D}_\mathbf{a}}(\boldsymbol\beta) \star G^{\mathcal{D}_\mathbf{a}}(\boldsymbol\beta) .
\label{eq composition is matrix multiplication}
\end{equation}
\end{theorem}
\begin{proof}
Let $ \mathbf{y} = G^{\mathcal{D}_\mathbf{a}}(\boldsymbol\beta) \in \mathbb{F}^M $ and let $ \mathbf{y}^i \in K_\mathbf{a}^M $, for $ i = 1,2, \ldots, M $, be the unique vectors such that $ \mathbf{y} = \sum_{i=1}^M \beta_i \mathbf{y}^i $. Then (\ref{eq composition is matrix multiplication}) follows from
$$ F^{\mathcal{D}_\mathbf{a}}(\mathbf{y}) = \sum_{i=1}^M F^{\mathcal{D}_\mathbf{a}}(\beta_i) \mathbf{y}^i = F^{\mathcal{D}_\mathbf{a}}(\boldsymbol\beta) \star \mathbf{y}, $$
where the first equality follows from Lemma \ref{lemma lin pols are right linear}, and the second equality is (\ref{eq rewriting matrix-Schur product}).
\end{proof}

Combining Theorems \ref{th skew pol is composition} and \ref{th composition is matrix multiplication}, we deduce the following.

\begin{corollary} \label{cor canonical ring isomorphisms in one conjugacy class}
With notation and assumptions as above, and considering $ \mathbb{F}^M $ as a ring with product $ \star_{\boldsymbol\beta} $, we have the following chain of natual ring isomorphisms
$$ \mathbb{F}[\mathbf{x}; \sigma, \delta]/I(\mathcal{C}(\mathbf{a})) \stackrel{ \phi_\mathbf{a} }{ \longrightarrow } \mathbb{F}[\mathcal{D}_\mathbf{a}] \stackrel{ E^L_{\boldsymbol\beta} }{ \longrightarrow } \mathbb{F}^M \stackrel{ \mu_{\boldsymbol\beta} }{ \longrightarrow } K_\mathbf{a}^{M \times M} , $$
where $ \phi_\mathbf{a} $ is as in (\ref{eq isomorphism for artin}), and $ E^L_{\boldsymbol\beta} $ is as in Definition \ref{def lin evaluation}. In particular, the rings above are simple \cite[Def. (2.1)]{lam-book} by \cite[Th. (3.1)]{lam-book}.
\end{corollary}

We conclude with the following decomposition theorem.

\begin{theorem} \label{th product decompositions}
Let $ \mathbf{a}_1, \mathbf{a}_2, \ldots, \mathbf{a}_\ell \in \mathbb{F}^n $ be pair-wise non-conjugate, and define $ \Omega = \mathcal{C}(\mathbf{a}_1) \cup \mathcal{C}(\mathbf{a}_2) \cup \ldots \cup \mathcal{C}(\mathbf{a}_\ell) $. The maps
\begin{equation}
\begin{array}{ccccc}
 \mathbb{F}[\mathbf{x}; \sigma, \delta] / I(\Omega) & \longrightarrow & \bigoplus_{i=1}^\ell \mathbb{F}[\mathbf{x}; \sigma, \delta] / I(\mathcal{C}(\mathbf{a}_i)) & \longrightarrow &  \bigoplus_{i=1}^\ell \mathbb{F}[\mathcal{D}_{\mathbf{a}_i}] \\
F + I(\Omega) & \mapsto & \left( F + I(\mathcal{C}(\mathbf{a}_i)) \right)_{i=1}^\ell & \mapsto & \left( F^{\mathcal{D}_{\mathbf{a}_i}} \right)_{i=1}^\ell
\end{array}
\label{eq product decomposition for omega}
\end{equation}
are left $ \mathbb{F} $-linear ring isomorphisms. In particular, if $ \mathbb{F} $ has finite right dimension over $ K_{\mathbf{a}_i} $, for $ i = 1,2, \ldots, \ell $, then $ \mathbb{F}[\mathbf{x}; \sigma, \delta] / I(\Omega) $ is a semisimple ring \cite[Def. (2.5)]{lam-book} by \cite[(3.3)]{lam-book}, \cite[(3.4)]{lam-book} and Corollary \ref{cor canonical ring isomorphisms in one conjugacy class}. 
\end{theorem}
\begin{proof}
It follows by Corollary \ref{cor canonical ring isomorphisms in one conjugacy class} and $ I(\Omega) = I(\mathcal{C}(\mathbf{a}_1)) \cap I(\mathcal{C}(\mathbf{a}_2)) \cap \ldots \cap I(\mathcal{C}(\mathbf{a}_\ell)) $.
\end{proof}

\section{Generalizations of Vandermonde, Moore and Wronskian matrices} \label{sec generalizations Vandermonde, Moore, Wronskian}

One of the main objectives behind the results on evaluations of univariate skew polynomials in \cite{lam, lam-leroy} was to generalize the classical rank computations of Vandermonde \cite{lam}, Moore \cite{moore, orespecial} and Wronskian \cite[Def. 1.11]{singer} matrices. A general and explicit method for calculating their ranks was obtained by combining \cite[Th. 4.5]{lam-leroy} and \cite[Th. 23]{lam}, which amount to linearizing the concept of P-independence in the univariate case, as done in Section \ref{sec linearized P-closed sets} for the multivariate case (Theorems \ref{theorem linearized version of P-closed in one conj} and \ref{theorem lin P-closed sets in several conj classes}).

An extension of general Vandermonde matrices as in \cite{lam, lam-leroy} using skew evaluations of multivariate skew polynomials (Definition \ref{def skew evaluation}) was given in \cite[Def. 40]{multivariateskew}. See Definition \ref{def vandermonde skew}. In this section (Definition \ref{def lin vandermonde matrices}), we give a definition using linearized evaluations as in Definition \ref{def lin evaluation}, and relate both Vandermonde matrices in Theorem \ref{th transforming skew to lin vandermonde}. In Theorem \ref{th rank of lin vandermonde matrix}, we combine Theorems \ref{theorem linearized version of P-closed in one conj} and \ref{theorem lin P-closed sets in several conj classes} to provide a general and explicit method for finding the rank of multivariate Vandermonde matrices as in Definitions \ref{def vandermonde skew} and \ref{def lin vandermonde matrices}.

\begin{definition} [\textbf{Skew Vandermonde matrices \cite{lam, lam-leroy, multivariateskew}}] \label{def vandermonde skew}
Let $ \mathcal{N} \subseteq \mathcal{M} $ be a finite set of skew monomials and let $ \mathcal{B} = \{ \mathbf{b}_1, \mathbf{b}_2, \ldots, \mathbf{b}_M \} \subseteq \mathbb{F}^n $. We define the skew Vandermonde matrix $ V_{\mathcal{N}}(\mathcal{B}) $ as the $ | \mathcal{N} | \times M $ matrix formed by the rows
$$ (N_\mathfrak{m}(\mathbf{b}_1), N_\mathfrak{m}(\mathbf{b}_2), \ldots, N_\mathfrak{m}(\mathbf{b}_M)) \in \mathbb{F}^M, $$
for all $ \mathfrak{m} \in \mathcal{N} $ (given certain ordering in $ \mathcal{N} $). For a positive integer $ d $, it will be useful to define $ \mathcal{M}_d $ as the set of monomials of degree less than $ d $.
\end{definition}

The following result is \cite[Prop. 41]{multivariateskew}, and connects the rank of a skew Vandermonde matrix with the underlying P-closed set. This result is essentially a reformulation of skew polynomial Lagrange interpolation as in Theorem \ref{th lagrange interpolation}.

\begin{proposition}[\textbf{\cite{multivariateskew}}] \label{prop rank vandermonde}
Let $ \mathcal{G} \subseteq \mathbb{F}^n $ be a finite set with $ M $ elements, and define $ \Omega = \overline{\mathcal{G}} $. If $ \mathcal{M}_M \subseteq \mathcal{N} \subseteq \mathcal{M} $, with notation as in Definition \ref{def vandermonde skew}, then
$$ {\rm Rk} \left( V_{\mathcal{N}}(\mathcal{G}) \right) = {\rm Rk}(\Omega) . $$
\end{proposition}

Note that, when $ \sigma = {\rm Id} $ and $ \delta = 0 $, skew Vandermonde matrices as above recover classical multivariate Vandermonde matrices. However, in the univariate case $ n = 1 $, they do not recover matrices such as Moore or Wronskian matrices. We now define \textit{linearized (multivariate) Vandermonde matrices}, which in the univariate case $ n = 1 $ do recover Moore and Wronskian matrices. In the multivariate case, they provide multivariate generalizations of such matrices.

\begin{definition} [\textbf{Linearized Vandermonde matrices}] \label{def lin vandermonde matrices}
Let $ \mathcal{N} \subseteq \mathcal{M} $ be a finite set of skew monomials. Let $ \mathbf{a}_1, \mathbf{a}_2, \ldots, \mathbf{a}_\ell \in \mathbb{F}^n $ and $ \mathcal{B}_i^\mathcal{D} = \{ \beta_1^{(i)}, \beta_2^{(i)}, \ldots, \beta_{M_i}^{(i)} \} \subseteq \mathbb{F}^* $, for $ i = 1,2,\ldots, \ell $. Denote $ \mathbf{a} = (\mathbf{a}_1, \mathbf{a}_2, \ldots, \mathbf{a}_\ell) $, $ \mathcal{B}^\mathcal{D} = (\mathcal{B}_1^\mathcal{D}, \mathcal{B}_2^\mathcal{D}, \ldots, \mathcal{B}_\ell^\mathcal{D}) $ and $ M = M_1 + M_2 + \cdots + M_\ell $. We define the linearized Vandermonde matrix $ V^\mathcal{D}_\mathcal{N}(\mathbf{a}, \mathcal{B}^\mathcal{D}) $ as the $ | \mathcal{N} | \times M $ matrix formed by the rows
$$ \left( \mathcal{D}^\mathfrak{m}_{\mathbf{a}_1}(\mathcal{B}_1^\mathcal{D}) , \mathcal{D}^\mathfrak{m}_{\mathbf{a}_2}(\mathcal{B}_2^\mathcal{D}), \ldots, \mathcal{D}^\mathfrak{m}_{\mathbf{a}_\ell}(\mathcal{B}_\ell^\mathcal{D}) \right) \in \mathbb{F}^M, $$
for all $ \mathfrak{m} \in \mathcal{N} $ (given certain ordering in $ \mathcal{N} $), where we define
$$ \mathcal{D}^\mathfrak{m}_{\mathbf{a}_i}(\mathcal{B}_i^\mathcal{D}) = \left( \mathcal{D}^\mathfrak{m}_{\mathbf{a}_i}(\beta_1^{(i)}) , \mathcal{D}^\mathfrak{m}_{\mathbf{a}_i}(\beta_2^{(i)}), \ldots, \mathcal{D}^\mathfrak{m}_{\mathbf{a}_i}(\beta_{M_i}^{(i)}) \right) \in \mathbb{F}^{M_i}, $$
for $ i = 1,2, \ldots, \ell $. 
\end{definition}

By rewriting Theorem \ref{theorem evaluation of skew as map}, we may easily connect skew Vandermonde matrices (Definition \ref{def vandermonde skew}) and linearized Vandermonde matrices (Definition \ref{def lin vandermonde matrices}).

\begin{theorem} \label{th transforming skew to lin vandermonde}
With notation as in Definition \ref{def lin vandermonde matrices}, define $ \mathcal{B} = \{ \mathbf{b}_1, \mathbf{b}_2, \ldots, \mathbf{b}_M \} \subseteq \mathbb{F}^n $ by
\begin{equation}
\mathbf{b}_j^{(i)} = \mathbf{a}_i^{\beta_j^{(i)}} = \mathcal{D}_{\mathbf{a}_i}(\beta_j^{(i)}) (\beta_j^{(i)})^{-1} , 
\label{eq transforming basis to P-basis}
\end{equation}
where $ \mathbf{b}_j^{(i)} = \mathbf{b}_r $ and $ r = M_1 + M_2 + \cdots + M_{i-1} + j $, for $ j = 1, $ $2, $ $ \ldots, M_i $ and $ i = 1,2, \ldots, \ell $. Then, for any finite set of skew monomials $ \mathcal{N} \subseteq \mathcal{M} $, it holds that
\begin{equation}
V^\mathcal{D}_\mathcal{N}(\mathbf{a}, \mathcal{B}^\mathcal{D}) = V_\mathcal{N}(\mathcal{B}) \cdot {\rm diag} \left( \beta_1^{(1)}, \beta_2^{(1)}, \ldots , \beta_{M_\ell}^{(\ell)} \right) ,
\label{eq transforming skew to lin vandermonde matrix}
\end{equation}
using the same ordering of $ \mathcal{N} $ to order the rows in $ V_\mathcal{N}(\mathcal{B}) $ and $ V^\mathcal{D}_\mathcal{N}(\mathbf{a}, \mathcal{B}^\mathcal{D}) $.
\end{theorem}

The following is the second main result of this section. It combines Proposition \ref{prop rank vandermonde} and Theorems \ref{theorem linearized version of P-closed in one conj} and \ref{theorem lin P-closed sets in several conj classes} to find the rank of skew and linearized Vandermonde matrices.

\begin{theorem} \label{th rank of lin vandermonde matrix}
Let the notation be as in Definition \ref{def lin vandermonde matrices} and Theorem \ref{th transforming skew to lin vandermonde}. Assume moreover that $ \mathbf{a}_1, \mathbf{a}_2, \ldots, \mathbf{a}_\ell \in \mathbb{F}^n $ are pair-wise non-conjugate and $ \mathcal{M}_{|\mathcal{B}|} \subseteq \mathcal{N} \subseteq \mathcal{M} $. Then 
$$ {\rm Rk}(V^\mathcal{D}_\mathcal{N}(\mathbf{a}, \mathcal{B}^\mathcal{D})) = {\rm Rk}(V_\mathcal{N}(\mathcal{B})) = \sum_{i=1}^\ell \dim^R_{K_{\mathbf{a}_i}} \left( \left\langle \beta_1^{(i)}, \beta_2^{(i)}, \ldots, \beta_{M_i}^{(i)} \right\rangle^R_{K_{\mathbf{a}_i}} \right) . $$
\end{theorem}

Theorem \ref{th rank of lin vandermonde matrix} above recovers several well-known particular results. 

Consider the univariate case $ n = 1 $. When $ \delta = 0 $, $ \ell = 1 $ and $ a_1 = 1 $, Theorem \ref{th rank of lin vandermonde matrix} says that the rank of a classical Moore matrix equals the dimension of the vector space generated by the evaluation points over $ K_1 = \{ \beta \in \mathbb{F} \mid \sigma(\beta) = \beta \} $. Analogously, when $ \sigma = {\rm Id} $, $ \ell = 1 $ and $ a_1 = 0 $, Theorem \ref{th rank of lin vandermonde matrix} says that the rank of a classical Wronskian matrix equals the dimension of the vector space generated by the evaluation points over $ K_0 = \{ \beta \in \mathbb{F} \mid \delta(\beta) = 0 \} $. In this generality, we recover Amitsur's result \cite[Th. 2]{amitsur-diff} on the existence of a linear differential equation with a given vector space of solutions.

Now consider the multivariate case $ n \geq 1 $. When $ \sigma = {\rm Id} $ and $ \delta = 0 $, Theorem \ref{th rank of lin vandermonde matrix} states that a conventional multivariate Vandermonde matrix is right invertible if it contains all monomials of degree less than the size of the evaluation set. When $ \sigma = {\rm Id} $, $ \ell = 1 $ and $ \mathbf{a}_1 = \mathbf{0} $, Theorem \ref{th rank of lin vandermonde matrix} above recovers Roth's characterization of invertible multivariate Wronskian matrices \cite[Lemma 1]{roth-wronskian}. See also \cite[Lemma D.11]{singer}.

These particular cases of Theorem \ref{th rank of lin vandermonde matrix} are recovered thanks to linearized evaluations as in Definition \ref{def lin evaluation}, whereas they are not recovered if we only consider remainder-based evaluations (Definition \ref{def skew evaluation}). This is because multivariate Moore and Wronskian matrices are particular cases of linearized Vandermonde matrices (Definition \ref{def lin vandermonde matrices}), but they are not particular cases of skew Vandermonde matrices (Definition \ref{def vandermonde skew}).

\section{Generalizations of Galois-theoretic results} \label{sec generalizations galois and derivations}

In this section, we define P-Galois extensions of division rings, which recover classical Galois extensions of fields. We rephrase and extend Corollaries \ref{cor dimensions for artin}, \ref{cor canonical ring isomorphisms in one conjugacy class} and \ref{cor first version hilbert 90} so that we are able to generalize, respectively, three important results in (finite) Galois theory: Artin's Theorem, the Galois correspondence and Hilbert's Theorem 90. 

In \cite[Sec. VII-5]{jacobson-structure}, Jacobson defines Galois extensions of division rings as those where the division subring is the set of fixed elements of a finite group of automorphisms of the larger division ring. We use the same idea in the following definition.

\begin{definition} [\textbf{P-Galois extensions}] \label{def P-Galois}
Given a division ring $ \mathbb{F} $ and one of its division subrings $ K \subseteq \mathbb{F} $, we say that the pair $ K \subseteq \mathbb{F} $ is a P-Galois extension if there exists a positive integer $ n $, a ring morphism $ \sigma : \mathbb{F} \longrightarrow \mathbb{F}^{n \times n} $, a $ \sigma $-derivation $ \delta : \mathbb{F} \longrightarrow \mathbb{F}^n $ and a point $ \mathbf{a} \in \mathbb{F}^n $, such that $ K = K_\mathbf{a} $ is a centralizer with respect to the skew polynomial ring $ \mathbb{F}[\mathbf{x}; \sigma, \delta] $ as in Definition \ref{def centralizer}. We say that the extension $ K \subseteq \mathbb{F} $ is finite if the right dimension of $ \mathbb{F} $ over $ K $ is finite.
\end{definition}

If $ \mathcal{G} $ is a finite group of ring automorphisms of $ \mathbb{F} $ generated by $ \sigma_1, \sigma_2, \ldots, $ $ \sigma_n $, then Definition \ref{def P-Galois} recovers (finite) Galois extensions $ K_\mathbf{a} \subseteq \mathbb{F} $ by choosing $ \sigma = {\rm diag}(\sigma_1, \sigma_2, $ $ \ldots, $ $ \sigma_n) $, $ \delta = 0 $ and $ \mathbf{a} = \mathbf{1} $. Definition \ref{def P-Galois} includes further cases by Example \ref{ex wild example II}.

In this sense, we may rewrite Corollary \ref{cor dimensions for artin} to extend Artin's Theorem \cite[Th. 14]{artin-lectures}.

\begin{theorem} \label{th artin P-Galois}
A P-Galois extension $ K_\mathbf{a} \subseteq \mathbb{F} $ is finite if, and only if, $ \mathbb{F}[\mathcal{D}_\mathbf{a}] $ has finite left dimension over $ \mathbb{F} $, in which case the latter dimension coincides with $ \dim^R_{K_\mathbf{a}}(\mathbb{F}) $. In particular, $ K_\mathbf{a} \subseteq \mathbb{F} $ is finite if the set $ \mathcal{D}_\mathbf{a}^\mathcal{M} = \{ \mathcal{D}_\mathbf{a}^\mathfrak{m} \mid \mathfrak{m} \in \mathcal{M} \} $ is finite, and 
\begin{equation}
\dim^R_{K_\mathbf{a}}(\mathbb{F}) \leq \left| \mathcal{D}_\mathbf{a}^\mathcal{M} \right|.
\label{eq upper bound artin theorem general}
\end{equation}
\end{theorem}

Artin's Theorem is recovered in the case where $ \mathbb{F} $ is commutative, $ \mathcal{G} $ is a finite group of ring automorphisms of $ \mathbb{F} $ generated by $ \sigma_1, \sigma_2, \ldots, \sigma_n $, $ \sigma = {\rm diag}(\sigma_1, \sigma_2, $ $ \ldots, $ $ \sigma_n) $, $ \delta = 0 $ and $ \mathbf{a} = \mathbf{1} $. In that case, we have that $ \mathcal{D}_\mathbf{a}^\mathcal{M} = \mathcal{G} $ and Artin's Theorem states that $ \dim_K(\mathbb{F}) = |\mathcal{G}| $, where $ K $ is the subfield of $ \mathbb{F} $ of elements fixed by $ \mathcal{G} $. If $ \mathbb{F} $ is not commutative, Theorem \ref{th artin P-Galois} above states that $ \dim^R_K(\mathbb{F}) \leq |\mathcal{G}| $. In such a case, $ \dim^R_K(\mathbb{F}) $ equals the \textit{reduced order} of $ \mathcal{G} $, which may be smaller than $ |\mathcal{G}| $. See \cite[Sec. VII-5]{jacobson-structure}.

We now extend the classical Galois correspondence \cite[Th. 16]{artin-lectures}.

\begin{theorem} \label{th P-galois correspondence}
Let $ \mathbf{a}, \mathbf{b} \in \mathbb{F}^n $. If $ \mathbb{F}[\mathcal{D}_\mathbf{a}] = \mathbb{F}[\mathcal{D}_\mathbf{b}] $, then $ K_\mathbf{a} = K_\mathbf{b} $. Conversely, if $ K_\mathbf{a} = K_\mathbf{b} $ and $ K_\mathbf{a} \subseteq \mathbb{F} $ is finite, then $ \mathbb{F}[\mathcal{D}_\mathbf{a}] = \mathbb{F}[\mathcal{D}_\mathbf{b}] $. 
\end{theorem}
\begin{proof}
First, if $ \mathbb{F}[\mathcal{D}_\mathbf{a}] = \mathbb{F}[\mathcal{D}_\mathbf{b}] $, then $ K_\mathbf{a} = K_\mathbf{b} $ by Proposition \ref{prop centralizer is largest for linearity}. We now prove the reversed implication. Let $ M = \dim^R_{K_\mathbf{a}}(\mathbb{F}) < \infty $, denote $ K = K_\mathbf{a} = K_\mathbf{b} $ and let $ \boldsymbol\beta = (\beta_1, \beta_2, \ldots, \beta_M) \in \mathbb{F}^M $ be an ordered right basis of $ \mathbb{F} $ over $ K $. By Corollary \ref{cor canonical ring isomorphisms in one conjugacy class}, 
$$ \mu_{\boldsymbol\beta} \circ E^L_{\boldsymbol\beta} : \mathbb{F}[\mathcal{D}_\mathbf{a}] \longrightarrow K^{M \times M} \quad \textrm{and} \quad \mu_{\boldsymbol\beta} \circ E^L_{\boldsymbol\beta} : \mathbb{F}[\mathcal{D}_\mathbf{b}] \longrightarrow K^{M \times M} $$
are ring isomorphisms. Thus we have a ring isomorphism
$$ \psi = (E^L_{\boldsymbol\beta})^{-1} \circ E^L_{\boldsymbol\beta} : \mathbb{F}[\mathcal{D}_\mathbf{a}] \longrightarrow \mathbb{F}[\mathcal{D}_\mathbf{b}], $$
where the domain of $ E^L_{\boldsymbol\beta} $ is $ \mathbb{F}[\mathcal{D}_\mathbf{a}] $, and the domain of $ (E^L_{\boldsymbol\beta})^{-1} $ is $ \mathbb{F}[\mathcal{D}_\mathbf{b}] $. However, if $ F^{\mathcal{D}_\mathbf{a}} \in \mathbb{F}[\mathcal{D}_\mathbf{a}] $ and $ G^{\mathcal{D}_\mathbf{b}} = \psi (F^{\mathcal{D}_\mathbf{a}}) \in \mathbb{F}[\mathcal{D}_\mathbf{b}] $, then 
$$ E^L_{\boldsymbol\beta}(F^{\mathcal{D}_\mathbf{a}}) = E^L_{\boldsymbol\beta}(G^{\mathcal{D}_\mathbf{b}}) \in \mathbb{F}^M. $$
Since $ \boldsymbol\beta $ is an ordered right basis of $ \mathbb{F} $ over $ K $, by linearity it must hold that $ F^{\mathcal{D}_\mathbf{a}} = G^{\mathcal{D}_\mathbf{b}} $. Therefore $ \psi $ is the identity and we deduce that $ \mathbb{F}[\mathcal{D}_\mathbf{a}] = \mathbb{F}[\mathcal{D}_\mathbf{b}] $, as desired.
\end{proof}

For $ \mathbf{a} \in \mathbb{F}^n $, define its support as $ I = {\rm Supp}(\mathbf{a}) = \{ i \in \{ 1,2, \ldots, n \} \mid a_i \neq 0 \} $, and let $ \mathcal{H}_I $ be the subgroup of $ \mathcal{G} $ generated by $ \{ \sigma_i \mid i \in I \} $. Then
\begin{equation*}
\begin{split}
K_\mathbf{a} & = \mathbb{F}^{\mathcal{H}_I} = \{ \beta \in \mathbb{F} \mid \tau(\beta) = \beta, \forall \tau \in \mathcal{H}_I \}, \textrm{ and} \\
\mathbb{F}[\mathcal{D}_\mathbf{a}] & = \mathbb{F}[\mathcal{H}_I] = \left\lbrace \sum_{ \tau \in \mathcal{H}_I} F_\tau \tau \mid F_\tau \in \mathbb{F}, \forall \tau \in \mathcal{H}_I \right\rbrace .
\end{split}
\end{equation*}
That is, $ K_\mathbf{a} $ and $ \mathbb{F}[\mathcal{D}_\mathbf{a}] $ are, respectively, the subfield of elements fixed by $ \mathcal{H}_I $ and the group ring of $ \mathcal{H}_I $ over $ \mathbb{F} $. Hence, if we set $ \mathcal{G} = \{ \sigma_1, \sigma_2, \ldots, \sigma_n \} $, then Theorem \ref{th P-galois correspondence} recovers Galois' original correspondence \cite[Th. 16]{artin-lectures}. Notice that we may recover Galois' original correspondence by using linearized evaluations (Definition \ref{def lin evaluation}) because they allow us to recover group rings as particular cases. This would not be the case if we only use remainder-based evaluations (Definition \ref{def skew evaluation}).

Finally, we may rewrite Corollary \ref{cor first version hilbert 90} to generalize Hilbert's Theorem 90.

\begin{theorem} \label{th general hilbert 90}
Let $ \mathbf{a}, \mathbf{b} \in \mathbb{F}^n $ be such that $ \mathbb{F} $ has finite right dimension over $ K_\mathbf{a} $. There exists $ \beta \in \mathbb{F}^* $ such that $ \mathbf{b} = \mathcal{D}_\mathbf{a}(\beta) \beta^{-1} $ if, and only if, $ F(\mathbf{b}) = 0 $, for all $ F \in I(\mathcal{C}(\mathbf{a})) $.
\end{theorem}

To recover the original theorems, we explore the generators of $ I(\mathcal{C}(\mathbf{a})) $.

\begin{lemma} \label{lemma preliminary finite Noether eqs}
Let $ \mathbf{a} \in \mathbb{F}^n $ be such that $ \mathbb{F} $ has finite right dimension over $ K_\mathbf{a} $. By Corollary \ref{cor dimensions for artin}, there is a finite set $ \mathcal{N} \subseteq \mathcal{M} $ such that $ \mathcal{D}_\mathbf{a}^\mathcal{N} $ is a left basis of $ \mathbb{F}[\mathcal{D}_\mathbf{a}] $. The set
$$ \mathcal{N}^c = \left\lbrace x_i \mathfrak{n} \in \mathcal{M} \mid 1 \leq i \leq n, \mathfrak{n} \in \mathcal{N} \cup \{1\}, x_i \mathfrak{n} \notin \mathcal{N} \right\rbrace $$
is the smallest subset $ \mathcal{N}^\prime \subseteq \mathcal{M} \setminus \mathcal{N} $ such that, if $ \mathfrak{m} \in \mathcal{M} \setminus \mathcal{N} $, then there exist $ \mathfrak{m}^\prime \in \mathcal{M} $ and $ \mathfrak{n}^\prime \in \mathcal{N}^\prime $ with $ \mathfrak{m} = \mathfrak{m}^\prime \mathfrak{n}^\prime $. Furthermore, there exist $ F^{\mathfrak{n}^\prime}_\mathfrak{n} \in \mathbb{F} $, for $ \mathfrak{n} \in \mathcal{N} $ and $ \mathfrak{n}^\prime \in \mathcal{N}^c $, such that
\begin{equation}
I(\mathcal{C}(\mathbf{a})) = \left( \left\lbrace \mathfrak{n}^\prime - \sum_{\mathfrak{n} \in \mathcal{N}} F^{\mathfrak{n}^\prime}_\mathfrak{n} \mathfrak{n} \in \mathbb{F}[\mathbf{x}; \sigma,\delta] \mid \mathfrak{n}^\prime \in \mathcal{N}^c \right\rbrace \right) .
\label{eq ideal for hilbert90 with concrete generators}
\end{equation}
Finally, it holds that $ | \mathcal{N}^c | \leq n (|\mathcal{N}| + 1) $, where equality may be attained.
\end{lemma}
\begin{proof}
The minimality of $ \mathcal{N}^c $ and the upper bound on its size are easy to see. Since $ \mathcal{D}_\mathbf{a}^\mathcal{N} $ is a left basis of $ \mathbb{F}[\mathcal{D}_\mathbf{a}] $, there exist $ F^{\mathfrak{m}}_\mathfrak{n} \in \mathbb{F} $ such that
\begin{equation}
\mathcal{D}_\mathbf{a}^{\mathfrak{m}} = \sum_{\mathfrak{n} \in \mathcal{N}} F^{\mathfrak{m}}_\mathfrak{n} \mathcal{D}_\mathbf{a}^\mathfrak{n},
\label{eq defining property of F_n^nprime}
\end{equation}
for $ \mathfrak{m} \in \mathcal{M} $ and $ \mathfrak{n} \in \mathcal{N} $, where $ F^{\mathfrak{m}}_\mathfrak{n} = \delta_{\mathfrak{m}, \mathfrak{n}} $ if $ \mathfrak{m} \in \mathcal{N} $. Let now $ I \subseteq \mathbb{F}[\mathbf{x}; \sigma, \delta] $ be the left ideal on the right-hand side of (\ref{eq ideal for hilbert90 with concrete generators}). Using Theorem \ref{theorem evaluation of skew as map}, it follows from (\ref{eq defining property of F_n^nprime}) that $ I \subseteq I(\mathcal{C}(\mathbf{a})) $. Therefore, there exists a canonical surjective left linear map
$$ \rho : \mathbb{F}[\mathbf{x}; \sigma, \delta] / I \longrightarrow \mathbb{F}[\mathbf{x}; \sigma, \delta] / I(\mathcal{C}(\mathbf{a})). $$
Now, it is easy to see that $ \dim^L_\mathbb{F}(\mathbb{F}[\mathbf{x}; \sigma, \delta] / I) = |\mathcal{N}| = |\mathcal{D}_\mathbf{a}^\mathcal{N}| $. Since $ |\mathcal{D}_\mathbf{a}^\mathcal{N}| = \dim^L_\mathbb{F}(\mathbb{F}[\mathcal{D}_\mathbf{a}]) = \dim^L_\mathbb{F}(\mathbb{F}[\mathbf{x}; \sigma, \delta] / I(\mathcal{C}(\mathbf{a}))) $ by Corollary \ref{cor dimensions for artin}, we conclude that $ \rho $ is a left vector space isomorphism. Hence $ I = I(\mathcal{C}(\mathbf{a})) $ and we are done.
\end{proof}

Therefore we may obtain a strengthening of Theorem \ref{th general hilbert 90} as follows.

\begin{theorem} \label{th general hilbert 90 finite generators}
Let $ \mathbf{a}, \mathbf{b} \in \mathbb{F}^n $ be such that $ \mathbb{F} $ has finite right dimension over $ K_\mathbf{a} $. Let $ \mathcal{N}, \mathcal{N}^c \subseteq \mathcal{M} $ be as in Lemma \ref{lemma preliminary finite Noether eqs}. There exists $ \beta \in \mathbb{F}^* $ such that $ \mathbf{b} = \mathcal{D}_\mathbf{a}(\beta) \beta^{-1} $ if, and only if, $ \mathbf{b} \in \mathbb{F}^n $ satisfies the following no more than $ n (|\mathcal{N}| + 1) $ equations:
$$ N_{\mathfrak{n}^\prime}(\mathbf{b}) = \sum_{\mathfrak{n} \in \mathcal{N}} F^{\mathfrak{n}^\prime}_\mathfrak{n} N_\mathfrak{n}(\mathbf{b}), \quad \textrm{for all } \mathfrak{n}^\prime \in \mathcal{N}^c. $$
\end{theorem}

Theorem \ref{th general hilbert 90 finite generators} recovers Noether's extension of Hilbert's Theorem 90 \cite{noether} in the case of classical Galois extensions. If $ \mathbb{F} $ is commutative, $ \mathcal{G} $ is a finite group of field automorphisms of $ \mathbb{F} $ generated by $ \sigma_1, \sigma_2, \ldots, \sigma_n $, $ K $ is the subfield of elements of $ \mathbb{F} $ fixed by $ \mathcal{G} $, $ \sigma = {\rm diag}(\sigma_1, \sigma_2, \ldots, \sigma_n) $, $ \delta = 0 $ and $ \mathbf{a} = \mathbf{1} $, then Theorem \ref{th general hilbert 90 finite generators} reads as follows.

\begin{corollary} [\textbf{\cite{noether}}] \label{cor noether hilbert90}
Let $ K \subseteq \mathbb{F} $ be a finite Galois extension of fields with Galois group $ \mathcal{G} $ generated by $ \sigma_1, \sigma_2, \ldots, \sigma_n $. For $ \mathbf{b} = (b_1, b_2, \ldots, b_n) \in (\mathbb{F}^*)^n $, there exists $ \beta \in \mathbb{F}^* $ such that $ b_i = \sigma_i(\beta) \beta^{-1} $, for all $ i = 1,2, \ldots, n $, if and only if, $ N_\mathfrak{m} (\mathbf{b}) = N_\mathfrak{n}(\mathbf{b}) $, whenever $ \mathfrak{m}(\sigma) = \mathfrak{n}(\sigma) $, the symbolic evaluations of $ \mathfrak{m}, \mathfrak{n} \in \mathcal{M} $ at $ (\sigma_1, \sigma_2, \ldots, \sigma_n) $, respectively.
\end{corollary}

Finally, we recover the original Theorem 90 for cyclic Galois extensions \cite[Th. 90]{hilbert}.

\begin{corollary} [\textbf{\cite{hilbert}}] \label{cor original hilbert90}
Let $ K \subseteq \mathbb{F} $ be a finite Galois extension of fields with Galois group $ \mathcal{G} $ generated by $ \sigma $. For $ b \in \mathbb{F}^* $, there exists $ \beta \in \mathbb{F}^* $ such that $ b = \sigma(\beta) \beta^{-1} $ if, and only if, $ N_{\mathbb{F}/K}(b) = \sigma^{m-1}(b) \sigma^{m-2}(b) \cdots \sigma(b) b = 1 $, where $ m = \dim_K(\mathbb{F}) $.
\end{corollary}

\section*{Acknowledgement}

The author gratefully acknowledges the support from The Independent Research Fund Denmark (Grant No. DFF-7027-00053B).

 
\bibliographystyle{plain}

{
\footnotesize

\begin{thebibliography}{10}

\bibitem{projective-pols}
S.~Abhyankar.
\newblock Projective polynomials.
\newblock {\em Proc. Amer. Math. Soc.}, 125(6):1643--1650, 1997.

\bibitem{amitsur-diff}
A.~S. Amitsur.
\newblock A generalization of a theorem on linear differential equations.
\newblock {\em Bull. Amer. Math. Soc.}, 54(10):937--941, 1948.

\bibitem{artin-lectures}
E.~Artin.
\newblock {\em Galois {T}heory}.
\newblock Notre Dame Mathematical Lectures, no. 2. University of Notre Dame,
  Notre Dame, Ind., second edition, 1944.

\bibitem{AugotSkewRM}
D.~Augot, A.~Couvreur, J.~Lavauzelle, and A.~Neri.
\newblock Rank-metric codes over arbitrary {G}alois extensions and rank
  analogues of {R}eed--{M}uller codes.
\newblock {\em SIAM J. Appl. Algebra Geometry}, 5(2):165--199, 2021.

\bibitem{skew-evaluation1}
D.~Boucher and F.~Ulmer.
\newblock Linear codes using skew polynomials with automorphisms and
  derivations.
\newblock {\em Des., Codes, Crypto.}, 70(3):405--431, 2014.

\bibitem{cai-field}
H.~Cai, Y.~Miao, M.~Schwartz, and X.~Tang.
\newblock A construction of maximally recoverable codes with order-optimal
  field size.
\newblock {\em IEEE Trans.\ Info.\ Theory}, pages 1--1, 2021.

\bibitem{cohn}
P.~M. Cohn.
\newblock {\em Free rings and their relations}.
\newblock London: Academic Press, 1971.

\bibitem{skewRM}
W.~Geiselmann and F.~Ulmer.
\newblock Skew {R}eed-{M}uller codes.
\newblock In {\em Rings, Modules and Codes}, volume 727, pages 107--116.
  Contemporary Mathematics, 2019.

\bibitem{gopi-field}
S.~Gopi and V.~Guruswami.
\newblock Improved maximally recoverable {LRC}s using skew polynomials.
\newblock 2020.
\newblock {P}reprint: arXiv:2012.07804.

\bibitem{hilbert}
D.~Hilbert.
\newblock {\em Die {T}heorie der algebraischen {Z}ahlk{\"o}rper}, volume~4.
\newblock Jahresbericht der Deutschen Mathematiker-Vereinigung, 1897.

\bibitem{jacobson-structure}
N.~Jacobson.
\newblock {\em Structure of rings}.
\newblock Providence: American Mathematical Society, 1956.

\bibitem{lam}
T.~Y. Lam.
\newblock A general theory of {V}andermonde matrices.
\newblock {\em Expositiones Mathematicae}, 4:193--215, 1986.

\bibitem{lam-book}
T.~Y. Lam.
\newblock {\em A First Course in Noncommutative Rings}, volume 131.
\newblock Gradate Texts Mathematics. Springer, New York, NY, 1991.

\bibitem{algebraic-conjugacy}
T.~Y. Lam and A.~Leroy.
\newblock Algebraic conjugacy classes and skew polynomial rings.
\newblock In {\em Perspectives in Ring Theory}, pages 153--203. Springer, 1988.

\bibitem{lam-leroy}
T.~Y. Lam and A.~Leroy.
\newblock Vandermonde and {W}ronskian matrices over division rings.
\newblock {\em J. Algebra}, 119(2):308--336, 1988.

\bibitem{lang}
S.~Lang.
\newblock {\em Algebra}, volume 211.
\newblock Gradate Texts Mathematics. Springer, Berlin, 2002.

\bibitem{lidl}
R.~Lidl and H.~Niederreiter.
\newblock {\em Finite Fields}, volume~20.
\newblock Encyclopedia of Mathematics and its Applications. Addison-Wesley,
  Amsterdam, 1983.

\bibitem{linearizedRS}
U.~Mart{\'i}nez-Pe{\~n}as.
\newblock Skew and linearized {R}eed–{S}olomon codes and maximum sum rank
  distance codes over any division ring.
\newblock {\em J. Algebra}, 504:587--612, 2018.

\bibitem{skew-class}
U.~Mart{\'\i}nez-Pe{\~n}as.
\newblock Classification of multivariate skew polynomial rings over finite
  fields via affine transformations of variables.
\newblock {\em Finite Fields App.}, 65:101687, 2020.

\bibitem{multivariateskew}
U.~Mart{\'i}nez-Pe{\~n}as and F.~R. Kschischang.
\newblock Evaluation and interpolation over multivariate skew polynomial rings.
\newblock {\em J. Algebra}, 525:111--139, 2019.

\bibitem{moore}
E.~H. Moore.
\newblock A two-fold generalization of {F}ermat's theorem.
\newblock {\em Bull. Amer. Math. Soc.}, 2(7):189--199, 1896.

\bibitem{noether}
E.~Noether.
\newblock Der {H}auptgeschlechtssatz f{\"u}r relativ-{G}aloissche
  {Z}ahlk{\"o}rper.
\newblock {\em Mathematische Annalen}, 108(1):411--419, Dec 1933.

\bibitem{orespecial}
O.~Ore.
\newblock On a special class of polynomials.
\newblock {\em Trans. Amer. Math. Soc.}, 35(3):559--584, 1933.

\bibitem{ore}
O.~Ore.
\newblock Theory of non-commutative polynomials.
\newblock {\em Annals of Mathematics (2)}, 34(3):480--508, 1933.

\bibitem{oxley}
J.~G. Oxley.
\newblock {\em Matroid theory}, volume~3.
\newblock Oxford University Press, USA, 2006.

\bibitem{roth-wronskian}
K.~F. Roth.
\newblock Rational approximations to algebraic numbers.
\newblock {\em Mathematika}, 2(1):1–20, 1955.

\bibitem{singer}
M.~van~der Put and M.~F. Singer.
\newblock {\em Galois Theory of Linear Differential Equations}, volume 328.
\newblock Gradate Texts Mathematics. Springer, Berlin, 2003.

\end{thebibliography}

}

\end{document}